\documentclass[10pt, a4paper]{amsart}

\RequirePackage[OT1]{fontenc}
\RequirePackage{amsthm,amsmath}
\RequirePackage[numbers]{natbib}
\RequirePackage[colorlinks,citecolor=blue,urlcolor=blue]{hyperref}

\usepackage{amsfonts}
\usepackage{graphicx}
\usepackage{setspace}

\numberwithin{equation}{section}
\theoremstyle{plain}
\newtheorem{thm}{Theorem}[section]

\usepackage{amsfonts,mathrsfs}
\theoremstyle{plain}
\newtheorem{cor}{Corollary}[section]
\newtheorem{lem}{Lemma}[section]
\newtheorem{prop}{Proposition}[section]
\newtheorem{rmk}{\textit{Remark}}

\newcommand{\D}{\displaystyle}
\newcommand{\non}{\nonumber}
\newcommand{\Pm}{\mathbb{P}}
\newcommand{\Em}{\mathbb{E}}
\newcommand{\halmos}{\vspace{3mm} \hfill $\Box$}
\allowdisplaybreaks

\newcommand{\wo}{\mathcal{W}^{(\omega)}}
\newcommand{\zo}{\mathcal{Z}^{(\omega)}}

\newcommand{\w}{\mathcal{W}}
\newcommand{\z}{\mathcal{Z}}

\usepackage{color}
\newcounter{num}
\usepackage{ulem}

\title{Fluctuations of Omega-killed spectrally negative L\'evy processes}

\author{Bo Li}
\address{
Department of Mathematics and LPMC\\
Nankai University\\
China}
\email{libo@nankai.edu.cn}
\author{Zbigniew Palmowski}
\address{
Faculty of Pure and Applied Mathematics\\
Wroc\l aw University of Science and Technology\\
Wyb. Wyspia\'nskiego 27, 50-370 Wroc\l aw, Poland
}
\email{zbigniew.palmowski@gmail.com}

\thanks{Supported in part by the project RARE-318984, a Marie Curie IRSES Fellowship within the 7th European Community Framework Programme.
Supported by the National Natural Science Foundation of China under the grant No. 11601243
and by the National Science Centre under the grant 2013/09/B/ST1/01778}

\date{\today}

\begin{document}

\begin{abstract}
In this paper we solve the exit problems for (reflected) spectrally negative L\'evy processes,
which are exponentially killed with a killing intensity dependent on the present state of the process and analyze respective resolvents.
All identities are given in terms of new generalizations of scale functions.
For the particular cases $\omega(x)=q$ and $\omega(x)=q \mathbf{1}_{(a,b)}(x)$, we obtain results for the classical exit problems and the Laplace transforms of the occupation times in a given interval, until first passage times, respectively.
Our results can also be applied to find the bankruptcy probability in the so-called Omega model, where bankruptcy occurs at rate $\omega(x)$ when the L\'evy surplus process is at level $x<0$.
Finally, we apply the these results to obtain some exit identities for a spectrally positive self-similar Markov processes.
The main method throughout all the proofs relies on the classical fluctuation identities for L\'evy processes, the Markov property and some basic properties of a Poisson process.

\vspace{3mm}


\end{abstract}

\maketitle

\pagestyle{myheadings} \markboth{\sc B.\ Li --- Z.\ Palmowski} {\sc
Fluctuations of Omega-killed SNLP}

\vspace{1.8cm}

\tableofcontents

\newpage

\section{Introduction}

Exit problems for (reflected) spectrally negative L\'{e}vy
processes have been the object of several studies over the last 40
years and have been used in many applied fields, such as mathematical finance, 
risk and queuing theory, biology, physics and many others.
An overview can be found in \cite{Kyprianou2014:book:levy}.
The principal tools of analysis are based on the Wiener--Hopf factorisation,
It\^{o}'s excursion theory
and the martingale theory.
The aim of this paper is two-fold. 
Firstly, to generalize known exit identities to a $\omega$-killed version, for which previous results are special cases, 
and secondly, to derive the exit identities using an alternative method, based on the properties of non-homogeneous Poisson processes and the Markov property.

Similar exit problems with functional discounting have been previously considered, in the context of optimal stopping theory, 
for the case of diffusion processes in Beibel and Lerche \cite{Beibel} and Dayanik \cite{Dayanik} (see also references therein).
This type of state-dependent killing has also been analyzed in the context of
Feynman-Kac formula and option prices by Glau \cite{Glau}.
The results of this paper can be considered from a wider perspective when the law of the exit time, from a given set, 
is determined by the solution of some Dirichlet problems based on a
Schr\"{o}dinger-type operator with the omega potential (describing external field acting on the particle).

Let $\omega:\mathbb{R}\to\mathbb{R}_{+}$ be a locally bounded nonnegative measurable function and
$X=\{X_t, t\geq0\}$ be a spectrally negative L\'evy process. We denote
its first passage times by:
\begin{equation}\label{firstpassagetimes}
\tau_z^{-}:=\inf\{t>0: X_t<z\}\quad\text{and}\quad\tau_c^{+}:=\inf\{t>0: X_t>c\}.
\end{equation}
Throughout the paper, the law of $X$, such that $X_0=x$, is denoted by $\Pm_x$ and the corresponding expectation by $\Em_x$. We will write $\Pm$ and $\Em$ when $x=0$.
Our main interest in this paper is deriving closed formulas for the occupation times, weighted by the $\omega$ function, considred up to some exit
times. In particular, for $x\in[0,c]$ we will identify
\begin{align*}
\mathcal{A}(x,c):=&\ \Em_x\left[\exp\left(-\int_0^{\tau_{c}^{+}} \omega(X_t)\,dt\right); \tau_c^{+}<\tau_0^{-}\right],\\
\mathcal{B}(x,c):=&\ \Em_x\left[\exp\left(-\int_0^{\tau_{0}^{-}} \omega(X_t)\,dt\right);  \tau_0^{-}<\tau_c^{+}\right].
\end{align*}
Applying a limiting argument to the above will produce the one-sided $\omega$-killed exit identities for the spectrally negative L\'evy process.
Similar results will be derived for a reflected process at running minimum and maximum.
Finally, respective resolvents are also identified.

It turns out that the identities can be characterised by two families of functions:
$\{\mathcal{W}^{(\omega)}(x), x\in \mathbb{R}\}$ and $\{\mathcal{Z}^{(\omega)}(x), x\in \mathbb{R}\}$,
which we will call $\omega$-scale functions and are defined uniquely as the solutions of the following equations:
\begin{align}
\mathcal{W}^{(\omega)}(x)=&\ W(x)+ \int_0^{x} W(x-y) \omega(y) \mathcal{W}^{(\omega)}(y)\,dy, \label{eqn:Ww}\\
\mathcal{Z}^{(\omega)}(x)=&\ 1+ \int_0^{x} W(x-y) \omega(y) \mathcal{Z}^{(\omega)}(y)\,dy,\label{eqn:Zw}
\end{align}
respectively, where $W(x)$ is a classical zero scale function defined formally in \eqref{scaleW}.

In the case of a constant $\omega$ function, i.e. $\omega(x)=q$, we will show that the $\omega$-scale functions reduce to the classical scale functions
$(\mathcal{W}^{(\omega)}(x), \mathcal{Z}^{(\omega)}(x))=(W^{(q)}(x),Z^{(q)}(x))$, producing the well-known exit identities for the two-sided exit problems for spectrally negative L\'evy processes (see \cite{Kyprianou2005:Martingale:rev, Kyprianou2014:book:levy}).

Taking $$\omega(x)=p+ q \mathbf{1}_{(a,b)}(x)$$
replicates the main result of \cite{Loeffen2014:occupationtime:levy}, where the Laplace transform of the occupation 
time in a given interval, until the first passage time, is studied.
In particular, we will show that in this case:
\begin{align*}
\mathcal{W}^{(\omega)}(x)=&\ \mathcal{W}_{a}^{(p,q)}(x)- q \int_{b}^{x} W^{(p)}(x-z)\mathcal{W}_{a}^{(p,q)}(z)\,dz,\\
\mathcal{Z}^{(\omega)}(x)=&\ \mathcal{Z}_{a}^{(p,q)}(x)- q\int_{b}^{x} W^{(p)}(x-z)\mathcal{Z}_{a}^{(p,q)}(z)\,dz,
\end{align*}
where the functions $(\mathcal{W}_{a}^{(p,q)}(z), \mathcal{Z}_{a}^{(p,q)}(z))$ are defined by \eqref{wronnie} and \eqref{zronnie},
which were used in \cite{Loeffen2014:occupationtime:levy} (see also \cite{Zhou2014:occupationtime:levy}).

The types of $\omega$-scale funtions given in this paper can produce
prices of various occupation-time-related derivatives, such as step options, double knock-out corridor options
and quantile options for a spectrally negative L\'evy model (see e.g. in \cite{Cai2010:occupationtime:jumpdiffusion, Guerin2014:occupationmeasure:levy}).
In particular, we analysed in detail the case when the function $\omega$ is a step function.
In addition, we will use the main results to calculate the probability of bankruptcy for an Omega-model with the L\'evy risk process analysed e.g. in \cite{Albrecher2011:optimal:omega, Gerber2012:omega, Landriault2011:occupationtime:levy}.
Finally, the $\omega$-scale functions identified for an exponential $\omega$ function give the solutions to the two-sided exit problems of a spectrally positive
self-similar Markov process (see e.g. \cite{Chaumont, survey}).

The remainder of this paper is organised as follows. In Section \ref{mainresults} we present our main results.
Section \ref{examples} is dedicated to the analysis of some particular cases.
Finally, in Sections \ref{sec:facts}, \ref{Proofs} and  \ref{Proofs: res} we give the proofs of all results.


\section{Main results}\label{mainresults}
\subsection{Basic notations}
Let $X=\{X_t, t\geq0\}$ be a spectrally negative L\'evy process, that is a stochastic process with independent and stationary increments without positive jumps.
We hereby exclude the case that X is the negative of a subordinator. Since the L\'evy process X has no positive jumps, its Laplace transform always exists and is given by
\[
\Em\left[\exp(\theta X_t)\right]=\exp( t\psi(\theta)), \quad \forall\, \theta\geq0.
\]
The function $\psi(\theta)$, known as the Laplace exponent of $X$, is a continuous and strictly convex function on $[0,\infty)$ given by the L\'evy-Khintchine formula
\begin{equation}
	\psi(\theta)=\frac{\sigma^{2}}{2}\,\theta^{2}+ \mu \theta
	+\int_{(-\infty, 0)}\bigl(e^{\theta x}-1-\theta
	x\,\mathbf{1}_{\{x>-1\}}\bigr)\,\Pi (d x),
\label{kyprianou-palmowski:LH}
\end{equation}
where $\mu \in \mathbb{R}$, $\sigma\geq 0$ and $\Pi $ is a measure on $\mathbb{R}_-$, such that $\int (1\wedge x^{2})\,\Pi (d x)<\infty $.
Further, $\psi(\theta)$ tends to infinity as $\theta$ tends to infinity, which allows us to define the right continuous inverse
of $\psi$ denoted by $\Phi$, that is $\Phi(\theta)=\inf\{s>0: \psi(s)>\theta\}$.
The reader is referred to Bertoin \cite{Bertoin96:book} and Kyprianou \cite{Kyprianou2014:book:levy} for an introduction to the theory of L\'evy processes.

The forthcoming discussions greatly rely on the so-called scale functions which play a vital role in the fluctuation identities for spectrally negative L\'evy processes; see \cite{Bertoin96:book, Kyprianou2014:book:levy, Hubalek2011:scalefunction:examples} for more details.
For $q\geq 0$,  the scale function $W^{(q)}$ is defined as a continuous and increasing function such that $W^{(q)}(x)=0$, for all $x<0$, and satisfies
\begin{equation}\label{scaleW}
\int_0^{\infty} e^{-\theta y} W^{(q)}(y)\,dy=\frac{1}{\psi(\theta)-q}\quad\text{for $\theta>\Phi(q)$}.
\end{equation}
In the following, we denote the zero scale function $W(x):=W^{(0)}(x)$.

Using the definition of the scale function $W^{(q)}(x)$ in \eqref{scaleW}, we define the related scale function $Z^{(q)}(x)$ by
\begin{equation}\label{scaleZ}
Z^{(q)}(x):=1+ q \int_0^{x}W^{(q)}(y)\,dy, \quad x\in \mathbb{R}.
\end{equation}

Throughout this paper, $*$ denotes the convolution on the non-negative half-line:
\[
f*h(x)=\int_{0}^{x} f(x-y)h(y)\,dy.
\]
In Section \ref{sec:facts} we will prove the following basic lemma.
\begin{lem}\label{lem:equation} Let $h$ and $\omega$ be locally bounded functions on $\mathbb{R}$. Then, the equation
\begin{equation}
H(x)=h(x)+ \int_0^{x} W(x-y) \omega(y) H(y)\,dy \label{lem:eqn:1}
\end{equation}
admits a unique locally bounded solution $H(x)=H^{(\omega)}(x)$, on $\mathbb{R}$, where we take $H(x)=h(x)$ for $x<0$.
For any fixed $\delta\geq0$, $H$ satisfies \eqref{lem:eqn:1} if and only if $H$ satisfies the following equation:
\begin{equation}
H(x)=h_\delta(x)+ \int_0^{x} W^{(\delta)}(x-y) (\omega(y)-\delta) H(y)\,dy,\label{lem:eqn:2}
\end{equation}
where $\D h_\delta(x)=h(x)+\delta W^{(\delta)}*h(x)$.
Moreover, if $\omega_1(x)<\omega_2(x)$ for all $x$, then $H^{(\omega_1)}(x)<H^{(\omega_2)}(x)$.
\end{lem}

From now on we will assume that $\omega$ is a non-negative and locally bounded function on $\mathbb{R}$, 
unless we explicitly assume (in addition) that $\omega$ is bounded.

Recall that, given $\omega$ on $\mathbb{R}$, the $\omega$-scale functions $\mathcal{W}^{(\omega)}(x)$ and $\mathcal{Z}^{(\omega)}(x)$
are given as the unique locally bounded solutions on $\mathbb{R}$ of equations
(\ref{eqn:Ww}) and (\ref{eqn:Zw}), respectively.
Note that, by Lemma \ref{lem:equation}, both functions are well-defined.

\subsection{Two-sided exit problems and one-sided downward problem}
Our first main result can be stated as follows.
\begin{thm}\label{thm:main} For $x\leq c$ and the first passage times $\tau_0^-$ and $\tau_c^+$ defined in \eqref{firstpassagetimes}, we have:
\begin{align}
\mathcal{A}(x,c)=&\ \Em_x\left[\exp\left(-\int_0^{\tau_{c}^{+}} \omega(X_t)\,dt\right); \tau_c^{+}<\tau_0^{-}\right]= \frac{\mathcal{W}^{(\omega)}(x)}{\mathcal{W}^{(\omega)}(c)},\label{ans:A}\\
\mathcal{B}(x,c)=&\ \Em_x\left[\exp\left(-\int_0^{\tau_{0}^{-}} \omega(X_t)\,dt\right);  \tau_0^{-}<\tau_c^{+}\right]= \mathcal{Z}^{(\omega)}(x)- \frac{\mathcal{W}^{(\omega)}(x)}{\mathcal{W}^{(\omega)}(c)} \mathcal{Z}^{(\omega)}(c).\label{ans:B}
\end{align}
\end{thm}

It is clear from Theorem \ref{thm:main} that $\mathcal{W}^{(\omega)}(c)$ and $\frac{\mathcal{Z}^{(\omega)}(c)}{\mathcal{W}^{(\omega)}(c)}$ are monotonic functions of $c$.
Taking the limit $c\to+\infty$, in Theorem \ref{thm:main}, produces the following corollary.
\begin{cor}\label{cor:one:down}
Let $\D c_{\mathcal{W}^{-1}(\infty)}=\lim_{c\to+\infty} \mathcal{W}^{(\omega)}(c)^{-1}$ and $\D c_{\mathcal{Z}/\mathcal{W}(\infty)}=\lim_{c\to+\infty} \frac{\mathcal{Z}^{(\omega)}(c)}{\mathcal{W}^{(\omega)}(c)}$. Then for all $x\geq 0$:
\begin{align*}
\Em_x\left[\exp\left(-\int_0^{\infty} \omega(X_t)\,dt\right); \tau_0^{-}=\infty\right]=&\  c_{\mathcal{W}^{-1}(\infty)} \mathcal{W}^{(\omega)}(x),\\
\Em_x\left[\exp\left(-\int_0^{\tau_0^{-}} \omega(X_t)\,dt\right); \tau_0^{-}<\infty\right]=&\  \mathcal{Z}^{(\omega)}(x)- c_{\mathcal{Z}/\mathcal{W}(\infty)}\mathcal{W}^{(\omega)}(x).
\end{align*}
Moreover, $c_{\mathcal{W}^{-1}(\infty)}>0$ if and only if $\D\int_0^{\infty} \omega(y)\,dy<\infty$ and $X_t\to\infty$ a.s.
\end{cor}

\begin{rmk}\rm D\"{o}ring and Kyprianou \cite{Doring2015:integral} derive similar conditions for the finiteness of perpetual integral for L\'evy processes.
\end{rmk}

Let us introduce the more general scale functions $\mathcal{W}^{(\omega)}(x,y)$ and  $\mathcal{Z}^{(\omega)}(x,y)$ on $\mathbb{R}\times\mathbb{R}$, defined as the solutions to the
following equations:
\begin{align}
\mathcal{W}^{(\omega)}(x,y)=&\ W(x-y) + \int_{y}^{x} W(x-z) \omega(z) \mathcal{W}^{(\omega)}(z,y)\,dz, \label{eqn:Ww:xy}\\
\mathcal{Z}^{(\omega)}(x,y)=&\ 1 + \int_{y}^{x} W(x-z) \omega(z) \mathcal{Z}^{(\omega)}(z,y)\,dz, \label{eqn:Zw:xy}
\end{align}
respectively.
\begin{rmk}\label{przesuniecie}\rm Note that, given $y_{0}\in\mathbb{R}$, $\mathcal{W}^{(\omega)}(\cdot+ y_{0},y_{0})$ and $\mathcal{Z}^{(\omega)}(\cdot+ y_{0},y_{0})$ are the scale functions with respect to the shifted omega function $\omega(\cdot+y_{0})$ in the sense of \eqref{eqn:Ww} and \eqref{eqn:Zw}.
\end{rmk}
Applying Theorem \ref{thm:main},
and shifting arguments generalises, the above result to exit identities for any interval $[y,z]$.

\begin{cor}\label{cor:2} For $z\leq x\leq y$,
\begin{align*}
\Em_{x} \left[\exp\left(-\int_0^{\tau^{+}_{y}} \omega(X_s)\,ds\right); \tau_{y}^{+}<\tau_{z}^{-}\right]=&\ \frac{\mathcal{W}^{(\omega)}(x,z)} {\mathcal{W}^{(\omega)}(y, z)},\\
\Em_{x} \left[\exp\left(-\int_0^{\tau^{-}_{z}} \omega(X_s)\,ds\right); \tau_{z}^{-}<\tau_{y}^{+}\right]=&\ \mathcal{Z}^{(\omega)}(x,z)-  \frac{\mathcal{W}^{(\omega)}(x,z)} {\mathcal{W}^{(\omega)}(y,z)} \mathcal{Z}^{(\omega)}(y,z).
\end{align*}
\end{cor}

\begin{rmk}\label{rmk:1}\rm
For $\delta\geq0$ we have:
\begin{equation}\label{falka}
W^{(\delta)}-W= \delta W^{(\delta)}* W \quad\text{and}\quad Z^{(\delta)}-Z= \delta W^{(\delta)}*Z.
\end{equation}
These identities can be checked by taking the Laplace transforms of both sides.
Applying Lemma \ref{lem:equation} and (\ref{falka}) one can verify that
$\mathcal{W}^{(\omega)}(x)$, $\mathcal{Z}^{(\omega)}(x)$, $\mathcal{W}^{(\omega)}(x,y)$ and $\mathcal{Z}^{(\omega)}(x,y)$ satisfy the following equations:
\begin{align}
\mathcal{W}^{(\omega)}(x)=&\ W^{(\delta)}(x)+ \int_0^{x} W^{(\delta)}(x-y) (\omega(y)-\delta) \mathcal{W}^{(\omega)}(y)\,dy,\label{eqn:Ww:2}\\
\mathcal{Z}^{(\omega)}(x)=&\ Z^{(\delta)}(x)+ \int_0^{x} W^{(\delta)}(x-y) (\omega(y)-\delta) \mathcal{Z}^{(\omega)}(y)\,dy,\label{eqn:Zw:2}\\
\mathcal{W}^{(\omega)}(x, y)=&\ W^{(\delta)}(x- y)+ \int_{y}^{x} W^{(\delta)}(x-z) (\omega(z)-\delta) \mathcal{W}^{(\omega)}(z, y)\,dz,\label{eqn:Ww:xy:2}\\
\mathcal{Z}^{(\omega)}(x, y)=&\ Z^{(\delta)}(x- y)+ \int_{y}^{x} W^{(\delta)}(x-z) (\omega(z)-\delta) \mathcal{W}^{(\omega)}(z, y)\,dz.\label{eqn:Zw:xy:2}
\end{align}
In particular, taking constant $\omega(x)=\delta$ gives $\mathcal{W}^{(\omega)}(x)= W^{(\delta)}(x)$, $\mathcal{Z}^{(\omega)}(x)=Z^{(\delta)}(x)$ and $\mathcal{W}^{(\omega)}(x,y)=W^{(\delta)}(x-y)$, $\mathcal{Z}^{(\omega)}(x,y)=\mathcal{Z}^{(\delta)}(x-y)$.
\end{rmk}

\begin{rmk}\label{dodatkowauwaga3}\rm
By considering the function $\widehat{\omega}(x)=\omega(-x)$ one derives the exit identities for the dual process
$\widehat{X}=-X$. The characteristics of $\widehat{X}$ will be indicated by the addition of a hat over the existing notation for the characteristics of $X$.
For example, $\widehat{\Pm}_{x}$ denotes the law of of $x-X$ under $\Pm$.
In particular,
\[
\widehat{\Em}_{y}\left[ \exp\left(-\int_{0}^{\tau_{z}^{-}} \omega(X_{t})\,dt\right); \tau_{z}^{-}\leq \tau_{x}^{+}\right]=\frac{\mathcal{W}^{(\omega)}(x, y)}{\mathcal{W}^{(\omega)}(x, z)}.
\]
\end{rmk}

\bigskip

In the next theorem we present the representation of $\omega$-type resolvents.
\begin{thm}\label{thm:resolvent}
Let $\mathcal{W}^{(\omega)}(\cdot,\cdot)$  be the solution to \eqref{eqn:Ww:xy}. For
$x,y\in[0,c]$,
\begin{align}\label{thm:resolvent:U}
U^{(\omega)}(x,dy):=&\ \int_0^{\infty} \Em_{x}\left[\exp\left(-\int_0^{t} \omega(X_s)\,ds\right); t<\tau_0^{-}\wedge\tau_c^{+}, X_t\in dy\right]\,dt\non\\
=&\ \left(\frac{\mathcal{W}^{(\omega)}(x)}{\mathcal{W}^{(\omega)}(c)} \mathcal{W}^{(\omega)}(c,y)- \mathcal{W}^{(\omega)}(x,y)\right) \,dy.
\end{align}
\end{thm}


\subsection{Exit problems for reflected processes}
In \cite{Pistorius2004:passagetime:reflect}
the exit identities were derived for the reflected spectrally negative L\'evy process.
Using a completely different method of the proof, we generalize all of the existing identities to the $\omega$-versions.
Formally, we define
\[
Y_t=X_t-I_t,\qquad \widehat{Y}_t=\widehat{X}_t-\widehat{I}_t\stackrel{d}{=}S_t-X_t,
\]
where $\D I_t:=\inf_{0\leq s\leq t}(0\wedge X_s)$ and $S_t:=\sup_{0\leq s\leq t}(0\vee X_s)$.
We will denote the first passage times of the reflected processes by:
\begin{equation}
T_c=\inf\{t\geq 0:Y_t>c\}, \quad \widehat{T}_c=\inf\{t\geq 0:\widehat{Y}_t>c\}.
\end{equation}
To simplify notations, we assume that $W$ has a continuous first derivative on $(0,\infty)$; see \cite{chan2011:scalefunction:smooth} for discussions on their smoothness.
We further assume, in this case, that $\omega$ is continuous.
From the definition of $\omega$-scale function, given in \eqref{eqn:Ww}, it follows that, in the case $W$ has continuous first derivative, then
$\mathcal{W}^{(\omega)}$
has a continuous first derivative on $(0,\infty)$.
Let us remark that in the case when $W$ does not have continuous first derivative
and $\omega(x+)$ is finite,
the derivative of the $\omega$-scale function $\mathcal{W}^{(\omega)}$
should be understood as the right-derivative.

\begin{thm}\label{thm:refl}
For $0\leq x\leq c$ we have:
\begin{align}
\mathcal{C}(x,c):=&\ \Em_x \left[\exp\left(-\int_0^{T_c} \omega(Y_t)\,dt\right)\right]\ = \frac{\mathcal{Z}^{(\omega)}(x)}{\mathcal{Z}^{(\omega)}(c)},\label{ans:C}\\
\widehat{\mathcal{C}}(x,c):=&\ \widehat{\Em}_{x}\left[\exp\left(-\int_0^{{T}_{c}} \omega(c- {Y}_t)\,dt\right)\right]=\mathcal{Z}^{(\omega)}(c-x)- \frac{\mathcal{Z}^{(\omega)\prime}(c)}{\mathcal{W}^{(\omega)\prime}(c)} \mathcal{W}^{(\omega)}(c-x).\label{ans:C:dual}
\end{align}
\end{thm}

Similar result concerning $\omega$-type resolvents could be also derived.
\begin{thm}\label{thm:resolv:refl}
The measure
\[
L^{(\omega)}(x, dy ):=\int_0^{\infty} \Em_{x}\left[\exp\left(-\int_0^{t} \omega(Y_s)\,ds\right); t<T_c, Y_{t}\in dy\right]\,dt
\]
is absolutely continuous with respect to the Lebesgue measure and a version of its density is given by:
\begin{align}
l^{(\omega)}(x,y)=&\ \frac{\mathcal{Z}^{(\omega)}(x)}{\mathcal{Z}^{(\omega)}(c)} \mathcal{W}^{(\omega)}(c,y)- \mathcal{W}(x,y),\qquad x,y\in [0,c).
\end{align}
Moreover, a version of the measure:
\[
\widehat{L}^{(\omega)}(x, dy ):=\int_0^{\infty} \widehat{\Em}_{x}\left[\exp\left(-\int_0^{t} \omega(c-{Y}_s)\,ds\right); t<{T}_c, {Y}_t\in dy\right]\,dt
\]
is given by $\widehat{l}^{(\omega)}(x,0) \delta_0(dy)+\widehat{l}^{(\omega)}(x,y) dy$ for $x,y\in [0,c)$ and
\begin{align}
\widehat{l}^{(\omega)}(x,y)=&\ \frac{\mathcal{W}^{(\omega)}(c-x)}{\mathcal{W}^{(\omega)\prime}(c)}\cdot \frac{\partial}{\partial x}\mathcal{W}^{(\omega)}(c,c-y)- \mathcal{W}^{(\omega)}(c-x, c-y),\\
\widehat{l}^{(\omega)}(x,0)=&\ \mathcal{W}^{(\omega)}(c-x)\mathcal{W}^{(\omega)}(0)/\mathcal{W}^{(\omega)\prime}(c).
\end{align}
\end{thm}

\begin{rmk}\rm
If $\omega(x)=q$ then, for $x\geq0$, we obtain the classical results:
\[
\mathcal{C}(x,c)= \frac{Z^{(q)}(x)}{Z^{(q)}(c)},\quad
\widehat{\mathcal{C}}(x,c)= Z^{(q)}(c-x)- q W^{(q)}(c-x) \frac{W^{(q)}(c)}{W^{(q)\prime}(c)}
\]
and the corresponding resolvents, given by:
\begin{align*}
U^{(\omega)}(x,dy)=&\ \left(\frac{W^{(q)}(x)}{W^{(q)}(c)} W^{(q)}(c-y)- W^{(q)}(x-y)\right)\,dy,\\
l^{(\omega)}(x,y)=&\ \frac{Z^{(q)}(x)}{Z^{(q)}(c)} W^{(q)}(c-y)- W^{(q)}(x-y),\\
\widehat{l}^{(\omega)}(x,y)=&\ \frac{W^{(q)}(c-x)}{W^{(q)\prime}(c)} W^{(q)\prime}(y)- W^{(q)}(y-x),\\
\widehat{l}^{(\omega)}(x,0)=&\ W^{(q)}(0) W^{(q)}(c-x)/W^{(q)\prime}(c);
\end{align*}
see \cite{Kyprianou2014:book:levy, Pistorius2004:passagetime:reflect} for details.
\end{rmk}

\bigskip

\subsection{One-sided upward problem}
To solve the one-sided upward problem we
have to assume additionally that:
\begin{equation}\label{addass}
\omega(x)=\phi\qquad \text{for all $x\leq 0$.}
\end{equation}
Let us define the function $\mathcal{H}^{(\omega)}$, on $\mathbb{R}$, satisfying the following equation
\begin{equation}
\mathcal{H}^{(\omega)}(x)= e^{\Phi(\phi) x}+ \int_0^{x} W^{(\phi)}(x-z)(\omega(z)-\phi) \mathcal{H}^{(\omega)}(z)\,dz.\label{eqn:cor:H}
\end{equation}

\begin{thm}\label{cor:one:up}
Assume that (\ref{addass}) holds. Then for $x,c\in\mathbb{R}$ and $x\leq c$,
\begin{equation}
\Em_{x}\left[\exp\left(-\int_0^{\tau_c^{+}} \omega(X_s)\,ds\right); \tau_c^{+}<\infty\right]=\frac{\mathcal{H}^{(\omega)}(x)}{\mathcal{H}^{(\omega)}(c)}.\label{eqn:cor:1}
\end{equation}
Moreover, for $x,y\leq c$,
\begin{align}
\Xi^{(\omega)}(x,dy):=&\ \int_0^{\infty} \Em_{x}\left[\exp\left(-\int_0^{t} \omega(X_s)\,ds\right); t<\tau_c^{+}, X_t\in dy\right]\,dt\non\\
=&\ \left(\frac{\mathcal{H}^{(\omega)}(x)}{\mathcal{H}^{(\omega)}(c)} \mathcal{W}^{(\omega)}(c,y)- \mathcal{W}^{(\omega)}(x,y)\right)\,dy.\label{Xi}
\end{align}
\end{thm}

\begin{rmk}\label{dodatkowauwaga}\rm
The $\omega$-type potential measure of $X$, without killing, can also be derived using the function $\mathcal{H}^{(\omega)}$. Assume, in addition to (\ref{addass}),
that $\omega(x)=q$, for $x\geq \Upsilon$ for some $\Upsilon, q\geq 0$. Then
\begin{align}
\Theta^{(\omega)}(x, dy):=&\ \int_0^{\infty} \Em_{x}\left[\exp\left(-\int_0^{t} \omega(X_s)\,ds\right); X_t\in dy\right]\,dt\non\\
=&\  \frac{e^{-\Phi(q) y} + \int_y^{\infty} e^{-\Phi(q) z}(\omega(z)-q) \mathcal{W}^{(\omega)}(z, y)\,dz}{\frac{q-\phi}{\Phi(q)-\Phi(\phi)}+ \int_0^{\infty} e^{-\Phi(q) z}(\omega(z)-q)\mathcal{H}^{(\omega)}(z)\,dz} \mathcal{H}^{(\omega)}(x) - \mathcal{W}^{(\omega)}(x,y),
\label{dodatkowauwaga2}\end{align}
where $\frac{q-\phi}{\Phi(q)-\Phi(\phi)}$ is understood as $\Phi'(\phi)^{-1}$, for the case $\phi=q$.

In order to prove this expression, we use the fact that $\omega(x)=q$, for $x\geq \Upsilon$, where $\Upsilon$ is some constant, and apply Lemma \ref{lem:equation}
to $\mathcal{H}^{(\omega)}(x)$  and $\mathcal{W}^{(\omega)}(x,y)$, which gives:
\begin{align*}
\mathcal{H}^{(\omega)}(c)=&\ e^{\Phi(\phi)c}+(q-\phi)\int_0^{c}e^{\Phi(\phi) (c-z)}W^{(q)}(z)\,dz\\
&+ \int_0^{c} W^{(q)}(c-z)(\omega(z)-q)\mathcal{H}^{(\omega)}(z)\,dz,\\
\mathcal{W}^{(\omega)}(c, y)=&\ W^{(q)}(c-y)+ \int_{y}^{c}W^{(q)}(c-z)(\omega(z)-q)\mathcal{W}^{(\omega)}(z,y)\,dz
\end{align*}
for $c\geq0$. Then, multiplying both equations through by $e^{-\Phi(q) c}$ and taking the limit as c $\rightarrow \infty$, we have 
\begin{eqnarray*}
\lefteqn{\lim_{c\to\infty}e^{-\Phi(q) c}\mathcal{H}^{(\omega)}(c)}\\
&&=\lim_{c\to\infty}\left[e^{(\Phi(\phi)-\Phi(q))c}\left(1+(q-\phi)\int_0^{c}e^{-\Phi(\phi) z}W^{(q)}(z)\,dz\right)\right.\non\\
&&\qquad \left. + \int_0^{c} e^{-\Phi(q) c}W^{(q)}(c-z)(\omega(z)-q)\mathcal{H}^{(\omega)}(z)\,dz\right]\non\\
&&=\Phi'(q)\left(\frac{q-\phi}{\Phi(q)-\Phi(\phi)}+ \int_0^{\infty} e^{-\Phi(q) z}(\omega(z)-q)\mathcal{H}^{(\omega)}(z)\,dz\right)
\end{eqnarray*}
and
\begin{eqnarray*}
\lefteqn{\lim_{c\to\infty}e^{-\Phi(q)c}\mathcal{W}^{(\omega)}(c,y)}\\
&&=\Phi'(q)\left(e^{-\Phi(q) y} + \int_y^{\infty} e^{-\Phi(q) z}(\omega(z)-q) \mathcal{W}^{(\omega)}(z, y)\,dz\right).
\end{eqnarray*}
This completes the proof of \eqref{dodatkowauwaga2}.
\end{rmk}

These theorems produce the basic framework for fluctuation theory of the state-dependent killed spectrally negative L\'evy process. In the next section, we will analyse some particular examples. For obvious reasons we will focus only on a non-constant rate function $\omega$.


\section{Examples}\label{examples}

\subsection{Occupation times of intervals until the first passage times}
\

Let $\omega(x)=p+ q \mathbf{1}_{(a,b)}(x)$ for $0<a<b$. 
Then, in this case, all the identities of interest can
be expressed in terms of the following functions:
\begin{align}
\mathcal{W}_{a}^{(p,q)}(x):=&\ W^{(p+q)}(x)- q \int_0^{a} W^{(p+q)}(x-y) W^{(p)}(y)\,dy\non\\
=&\ W^{(p)}(x)+ q \int_a^{x} W^{(p+q)}(x-y) W^{(p)}(y)\,dy, \label{wronnie}\\
\mathcal{Z}_{a}^{(p,q)}(x):=&\ Z^{(p+q)}(x)- q\int_0^{a} W^{(p+q)}(x-y) Z^{(p)}(y)\,dy \non\\
=&\ Z^{(p)}(x)+ q \int_a^{x} W^{(p+q)}(x-y) Z^{(p)}(y)\,dy,\label{zronnie}
\end{align}
and
\begin{align}
\mathcal{H}^{(p,q)}(x):=&\ e^{\Phi(p) x}\left(1+ q\int_0^{x} e^{-\Phi(p) y} W^{(p+q)}(y)\,dy\right); \label{hronnie}
\end{align}
see \cite{Guerin2014:occupationmeasure:levy, Loeffen2014:occupationtime:levy} for definitions.
Now, if we apply the above form of the $\omega$ function in equation \eqref{lem:eqn:1} of Remark \ref{rmk:1} and set $\delta = p$, we have 
\begin{align}
\mathcal{W}^{(\omega)}(x)=&\ W^{(p)}(x)+ q\int_{0}^{x} W^{(p)}(x-z) \mathbf{1}_{\{z\in[a,b]\}}\mathcal{W}^{(\omega)}(z)\,dz\label{eqn:exam:Ww:1}\\
=&\ W^{(p+q)}(x)- q \int_0^{x} W^{(p+q)}(x-z) \mathbf{1}_{\{z\notin(a,b)\}} \mathcal{W}^{(\omega)}(z)\,dz.\non
\end{align}
Hence
\[
\mathcal{W}^{(\omega)}(x)= \left\{
\begin{array}{l@{\quad\text{for\ }}l}
W^{(p)}(x) & x\in[0,a],\\
\D W^{(p+q)}(x)- q \int_0^{a} W^{(p+q)}(x-z)W^{(p)}(z)\,dz & x\in[a,b].
\end{array}\right.
\]
Comparing this identity with the definition of $\mathcal{W}_{a}^{(p,q)}(x)$ in \eqref{wronnie} gives:
\[
\mathcal{W}^{(\omega)}(x)=\mathcal{W}_{a}^{(p,q)}(x)\quad \text{for\ }x\in[0,b].
\]
Then, substituting above equality back into \eqref{eqn:exam:Ww:1} produces:
\[
\mathcal{W}_{a}^{(p,q)}(x)= W^{(p)}(x)+ q\int_a^{x} W^{(p)}(x-z) \mathcal{W}_{a}^{(p,q)}(z)\,dz\quad \text{for}\ x\in[0,b].
\]
In fact, this identity holds for all $x\geq0$ and hence, by \eqref{eqn:exam:Ww:1}, we have:
\begin{equation}\label{piertoz}
\mathcal{W}^{(\omega)}(x)= \mathcal{W}_{a}^{(p,q)}(x)- q \int_b^{x} W^{(p)}(x-z) \mathcal{W}_{a}^{(p,q)}(z)\,dz \quad \text{for all $x\geq0$}.
\end{equation}
Similar arguments give:
\begin{equation}\label{drtoz}
\mathcal{Z}^{(\omega)}(x)= \mathcal{Z}_{a}^{(p,q)}(x)- q \int_b^{x} W^{(p)}(x-z) \mathcal{Z}_{a}^{(p,q)}(z)\,dz  \quad \text{for all $x\geq0$}.
\end{equation}

Combining identities (\ref{piertoz}) - (\ref{drtoz}) and Theorem \ref{thm:main} reproduces
the main results of \cite{Loeffen2014:occupationtime:levy}, which give the law of occupation times in a given interval. 
That is, for $0\leq a\leq b\leq c$ and $p, q>0$,
\begin{align*}
\Em_x\left[e^{-p\tau_c^{+}-q\int_0^{\tau_c^{+}} \mathbf{1}_{(a,b)}(X_s)\,ds}; \tau_c^{+}<\tau_0^{-}\right]=&\  \frac{\mathcal{W}^{(\omega)}(x)}{\mathcal{W}^{(\omega)}(c)},\\
\Em_x\left[e^{-p\tau_0^{-}-q\int_0^{\tau_0^{-}} \mathbf{1}_{(a,b)}(X_s)\,ds}; \tau_0^{-}<\tau_c^{+}\right]=&\ \mathcal{Z}^{(\omega)}(x)- \frac{\mathcal{W}^{(\omega)}(x)}{\mathcal{W}^{(\omega)}(c)} \mathcal{Z}^{(\omega)}(c).
\end{align*}

If we now consider the particular case $\omega(x)=p+q\mathbf{1}_{(0,b)}(x)$ (that is $a=0$) one can prove, as before, that:
\begin{align}
\mathcal{H}^{(\omega)}(x)=&\ \mathcal{H}^{(p,q)}(x)- q\int_{b}^{x} W^{(p)}(x-y) \mathcal{H}^{(p,q)}(y)\,dy, \\
\mathcal{W}^{(\omega)}(x,y)=&\ \mathcal{W}_{-y}^{(p,q)}(x-y)- q \int_{b}^{x} W^{(p)}(x-z) \mathcal{W}_{-y}^{(p,q)}(z-y)\,dz.
\end{align}
Then, by Theorem \ref{cor:one:up} we obtain Corollary \cite[Cor. 2]{Loeffen2014:occupationtime:levy}. That is, for $x\leq c$,
\[
\Em_{x}\left[e^{-p\tau_c^{+}- q\int_0^{\tau_{c}^{+}} \mathbf{1}_{(0,b)}(X_s)\,ds}; \tau_c^{+}<\infty\right]=\frac{\mathcal{H}^{(\omega)}(x)}{\mathcal{H}^{(\omega)}(c)}.
\]
Moreover,
\begin{align}
\Xi^{(\omega)}(x,dy)=&\ \frac{\mathcal{H}^{(p,q)}(x)- q\int_{b}^{x} W^{(p)}(x-z) \mathcal{H}^{(p,q)}(z)\,dz}{\mathcal{H}^{(p,q)}(c)- q\int_{b}^{c} W^{(p)}(c-z) \mathcal{H}^{(p,q)}(z)\,dz}\non\\
&\ \times \left(\mathcal{W}_{-y}^{(p,q)}(c-y)- q \int_{b}^{c} W^{(p)}(c-z) \mathcal{W}_{-y}^{(p,q)}(z-y)\,dz\right)\,dy\non\\
&\ \quad- \left(\mathcal{W}_{-y}^{(p,q)}(x-y)- q \int_{b}^{x} W^{(p)}(x-z) \mathcal{W}_{-y}^{(p,q)}(z-y)\,dz\right)\,dy.
\end{align}
Similarly, from (\ref{dodatkowauwaga2}):
\begin{align}
\Theta^{(\omega)}(x,dy)=&\ \frac{e^{-\Phi(p) y}+ q\int_{0}^{b} e^{-\Phi(p) z} \mathcal{W}_{-y}^{(p,q)}(z-y)\,dz}{\Phi'(p)^{-1}+ q\int_{0}^{b} e^{-\Phi(p)z} \mathcal{H}^{(p,q)}(z)\,dz}\non\\
&\ \times \left(\mathcal{H}^{(p,q)}(x)- q\int_{b}^{x} W^{(p)}(x-z) \mathcal{H}^{(p,q)}(z)\,dz\right)\,dy\non\\
&\ \quad- \left(\mathcal{W}_{-y}^{(p,q)}(x-y)- q \int_{b}^{x} W^{(p)}(x-z) \mathcal{W}_{-y}^{(p,q)}(z-y)\,dz\right)\,dy.
\end{align}

In fact, one can prove additional identities:
\begin{eqnarray*}
\lefteqn{\mathcal{W}_{x-b}^{(p,q)}(x-y)- q\int_0^{-y} \mathcal{W}_{x-b}^{(p,q)}(x-y-z){W}^{(p)}(z)\,dz}\\&&
=\mathcal{W}_{-y}^{(p,q)}(x-y)- q\int_{b}^{x} W^{(p)}(x-z) \mathcal{W}_{-y}^{(p,q)}(z-y)\,dz
\end{eqnarray*}
and
\begin{eqnarray*}
\lefteqn{e^{-\Phi(p) b} \left(\mathcal{H}^{(p,q)}(b-y)- q\int_0^{-y} W^{(p)}(z) \mathcal{H}^{(p,q)}(b-y-z)\,dz\right)}\\
&&=e^{-\Phi(p) y} + q \int_0^{b} e^{-\Phi(p)z} \mathcal{W}_{-y}^{(p,q)}(z-y)\,dz,
\end{eqnarray*}
which allow us to reproduce the results of 
\cite[Thm. 1]{Guerin2014:occupationmeasure:levy}.
All other applications of the above identities can be found in 
\cite{Guerin2014:occupationmeasure:levy, Kyprianou2013:occupationtime:Rlevy, Landriault2011:occupationtime:levy, Li2013:occupationtime:diffusion, Loeffen2014:occupationtime:levy}.

\subsection{Omega-model}\

In actuarial science one could consider a spectrally negative L\'{e}vy process to model the surplus of an insurer and $\omega$ as a bankruptcy function 
(disappearing on the positive half-line).
For $x<0$ the quantity $\omega(x)dt$ describes the probability of bankruptcy within $dt$ time units, which
distinguishes the technical ruin (which happens at $\tau_0^-$) from the bankruptcy; see Gerber et al. \cite{Gerber2012:omega}.
In addition, bankruptcy may also occur if the risk process falls below some fixed level $-d<0$.
Then, the bankruptcy probability is defined by:
\begin{equation}
\varphi(x)=1-
\Em_x\left[e^{-\int_0^{\infty}\omega(X_s)\,ds}; \tau_{-d}^-=\infty\right].
\end{equation}

For this model, we consider the case of $X$ being a linear Brownian motion, that is
\begin{equation}
X_t=x+\sigma B_t +\mu t,
\end{equation}
for some $x,\sigma,\mu>0$. Further, we take the specific bankruptcy function
\[
\omega(x)=
\left[\gamma_0+ \gamma_1 (x+d)\right]\mathbf{1}_{\{x\in [-d,0]\}}.
\]
It is straightforward to check, form the definition of the scale function, that
\[
W(y)=\frac{1}{\mu}(1-e^{-\frac{2\mu}{\sigma^{2}}y}) \quad\text{and}\quad W^{(\gamma_{0})}(y)=\frac{2}{\sigma^{2}\rho}(e^{\rho_{2} y}-e^{-\rho_{1} y})
\]
with $\rho_{1}-\rho_{2}=\frac{2\mu}{\sigma^{2}}$, $\rho_{1}\rho_{2}=\frac{2\gamma_{0}}{\sigma^{2}}$ and $\rho=\rho_{1}+\rho_{2}=\frac{2\sqrt{\mu^{2}+ 2\gamma_{0}\sigma^{2}}}{\sigma^{2}}$.
Then, applying the shifting arguments given in Remark \ref{przesuniecie}, from Corollary \ref{cor:one:down} we have:
\begin{align}\label{ruinprob}
\varphi (x)=&1-\ c_{\mathcal{W}^{-1}(\infty,-d)} \mathcal{W}^{(\omega)}(x,-d),
\end{align}
where in above equation we denoted $\D c_{\mathcal{W}^{-1}(\infty,-d)}=\left[\lim_{c\to\infty} \mathcal{W}^{(\omega)}(c,-d)\right]^{-1}$ and,
by \eqref{eqn:Ww:xy} and Remark \ref{rmk:1}, the function $\mathcal{W}^{(\omega)}(x,-d)$  satisfies the following equation for all $x\in \mathbb{R}$:
\begin{align}
\mathcal{W}^{(\omega)}(x,-d)
=&\ W(x+d)+ \int_{0}^{x+d} W(x+d-y) \omega(y-d) \mathcal{W}^{(\omega)}(y-d,-d)\,dy.
\label{eqn:exam:2:1}
\end{align}
Thus for $x \in[-d,0]$ we have:
\begin{align}
\mathcal{W}^{(\omega)}(x,-d)
=&\ W^{(\gamma_{0})}(x+d)+ \gamma_{1} \int_{0}^{x+ d} W^{(\gamma_{0})}(x+d-y)  y \mathcal{W}^{(\omega)}(y-d,-d)\,dy
.\label{eqn:exam:2:2}
\end{align}

Let us take the change of variable $z=x+d\geq 0$ and denote $g(z):=\mathcal{W}^{(\omega)}(z-d,-d)=\mathcal{W}^{(\omega)}(x,-d)$.
Now, observing that $(\frac{d}{dz}-\rho_{2})(\frac{d}{dz}+\rho_{1}) W^{(\gamma_{0})}(z)=0$, 
from equation \eqref{eqn:exam:2:2}, for $z\in [0,d]$ (hence for $x\in [-d,0]$), 
and \cite[Lem. 8.5, Problem 8.6]{Kyprianou2014:book:levy}, it follows that
\[
\left(\frac{d}{dz}-\rho_{2}\right)\left(\frac{d}{dz}+\rho_{1}\right)g(z)=
\frac{2\gamma_{1}}{\sigma^{2}} z g(z)
\]
with the boundary conditions $g(0)=0$
and $g'(0)=2/\sigma^{2}$.
Then, from \cite{Gerber2012:omega} one can conclude that for $x\in[-d,0]$,
\begin{equation}
\mathcal{W}^{(\omega)}(x,-d)= e^{- \frac{\mu (x+d)}{\sigma^{2}}}\left[
m_{1} {\rm Ai} \left(\sqrt[3]{\frac{2 \gamma_{1}}{\sigma^{2}}} \left(x+d+ \frac{\sigma^{2} \rho^{2}}{8\gamma_{1}}\right)\right)+ m_{2} {\rm Bi} \left(\sqrt[3]{\frac{2 \gamma_{1}}{\sigma^{2}}} \left(x+d+ \frac{\sigma^{2} \rho^{2}}{8\gamma_{1}}\right)\right)\right],\label{idw-d}
\end{equation}
where $m_{1}$ and $m_{2}$ are constants chosen in such a way that the boundary condition holds and ${\rm Ai}(x)$, ${\rm Bi}(x)$ denote the Airy and Bairy functions, respectively.

In order to identify $g(z)$, for $z\geq d$ (hence $x\geq 0$), note that from \eqref{eqn:exam:2:1} we have
\begin{equation}\label{pochodnap}
g'(z)=\frac{2}{\sigma^{2}} e^{-\frac{2\mu}{\sigma^{2}} z}+ \frac{2}{\sigma^{2}} \int_{0}^{z} e^{-\frac{2\mu}{\sigma^{2}} (z-y)} \omega(y-d) g(y)\,dy
\end{equation}
and hence
\begin{align*}
g(z)
=&\ \frac{1-e^{-\frac{2\mu}{\sigma^{2}}d}}{\mu}+ \int_{0}^{d} \left(\frac{1-e^{-\frac{2\mu}{\sigma^{2}}(d-y)}}{\mu}\right)\omega(y-d) g(y)\,dy\\
&\ + \frac{e^{-\frac{2\mu}{\sigma^{2}}d}- e^{-\frac{2\mu}{\sigma^{2}}z}}{\mu} \left(1+ \int_{0}^{d} e^{\frac{2\mu}{\sigma^{2}} y} \omega(y-d) g(y)\,dy\right).
\end{align*}
Thus, it follows that, for $x\geq 0$:
\begin{align*}
\mathcal{W}^{(\omega)}(x,-d)
=&\ g(d)+ \left(1- e^{-\frac{2\mu}{\sigma^{2}}x}\right) \frac{\sigma^{2}}{2\mu} g'(d),
\end{align*}
where $g(d)=\mathcal{W}^{(\omega)}(0,-d)$ and
$g'(d)=\frac{\partial}{\partial x}\mathcal{W}^{(\omega)}(x,-d)|_{x=0}$ for $\mathcal{W}^{(\omega)}(x,-d)$ given in (\ref{idw-d}).
Moreover, we have: $$c_{\mathcal{W}^{-1}(\infty,-d)}^{-1}= g(d)+ \frac{\sigma^{2}}{2\mu} g'(d).$$
Observe that for $x>0$ the bankruptcy probability is proportional to the classical ruin probability for the linear Brownian motion:
\[\varphi (x)=e^{-\frac{2\mu}{\sigma^{2}}x} \frac{\sigma^{2}}{2\mu} c_{\mathcal{W}^{-1}(\infty,-d)}g'(d)
=e^{-\frac{2\mu}{\sigma^{2}}x} \frac{  \frac{\sigma^{2}}{2\mu}g'(d)}{g(d)+\frac{\sigma^{2}}{2\mu} g'(d)}.
\]

\subsection{Exponential $\omega$ function and self-similar processes}

Our main motivation for considering function $\omega$ of this type 
comes from solving the exit problems for a positive self-similar Markov process (pssMp)
$\{\iota_t, t\geq 0\}$ with self-similarity parameter $\xi>0$ and associated spectrally positive
L\'evy process $\hat{X}_t=-X_t$. Note that, $\iota_t$ has no negative jumps; see \cite{survey} for details.
For $b<0<c$ we define the times
\[
\sigma_{b}^{+}:=\inf\{t>0: \iota_t>e^{-b}\}\quad\text{and}\quad\sigma_{c}^{-}=\inf\{t>0: \iota_t<e^{-c}\},
\]
where we take $\iota_0=1$. Using the following Lamperti  transformation (see e.g. \cite{survey})
\[\iota_t=e^{-X_{\tau(t)}}\quad
\text{for}\quad \tau(t)=\inf\left\{s\geq 0: \int_0^s e^{-\xi X_w}\,dw\geq t\right\},\]
from
Corollary \ref{cor:2} and Remark \ref{przesuniecie} we have:
\begin{align*}
\Em_1 \left[e^{-\varrho_{0}\sigma_{c}^{-}}; \sigma_{c}^{-}<\sigma_{b}^{+}\right]=&\
\Em \left[\exp\left(-\varrho_{0} \int_0^{\tau^{+}_{c}} e^{-\xi X_s}\,ds\right); \tau_{c}^{+}<\tau_{b}^{-}\right]=
\frac{\mathcal{W}^{(\omega)}(-b)} {\mathcal{W}^{(\omega)}(c-b)},\\
\Em_1 \left[e^{-\varrho_{0}\sigma_{b}^{+}}; \sigma_{b}^{+}<\sigma_{c}^{-}\right]=&\
\Em \left[\exp\left(-\varrho_{0}\int_0^{\tau^{-}_{b}} e^{-\xi X_s}\,ds\right); \tau_{b}^{-}<\tau_{c}^{+}\right]
\\=&\
\mathcal{Z}^{(\omega)}(-b)-  \frac{\mathcal{W}^{(\omega)}(-b)} {\mathcal{W}^{(\omega)}(c-b)} \mathcal{Z}^{(\omega)}(c-b),
\end{align*}
for
\begin{equation}\label{expoomega}
\omega(x)=\varrho\cdot e^{-\xi x},
\end{equation}
with $$\varrho=\varrho_{0} e^{\xi b}.$$
If one considers the above identities with $\varrho=0$, we obtain the identities given in \cite[Lemma 1]{Chaumont}.

We will now identify the $\omega$-scale functions $\mathcal{W}^{(\omega)}$ and $\mathcal{Z}^{(\omega)}$, for $\omega$ given in
(\ref{expoomega}).
For the above case, these functions are the unique solutions satisfying the following equations:
\begin{align*}
\mathcal{W}^{(\omega)}(x)=&\ W(x)+ \varrho \int_0^{x} W(x-y) e^{-\xi y} \mathcal{W}^{(\omega)}(y)\,dy,\\
\mathcal{Z}^{(\omega)}(x)=&\ 1+ \varrho \int_0^{x} W(x-y) e^{-\xi y} \mathcal{Z}^{(\omega)}(y)\,dy.
\end{align*}
Now, by taking Laplace transforms on both sides of the above equations we derive, for $s>\Phi(0)$:
\[
\widehat{\mathcal{W}^{(\omega)}}(s)=\psi(s)^{-1} (1+ \varrho\widehat{\mathcal{W}^{(\omega)}}(s+\xi)) \text{\quad and\quad } \widehat{\mathcal{Z}^{(\omega)}}(s)=s^{-1}+ \varrho \psi(s)^{-1} \widehat{\mathcal{Z}^{(\omega)}}(s+\xi).
\]
Then, it follows that 
\begin{align}
\widehat{\mathcal{W}^{(\omega)}}(s)=&\ \frac{1}{\varrho}\sum_{n=0}^{\infty}\prod_{0\leq k\leq n}  \left(\frac{\varrho}{\psi(s+ k\xi)}\right),\\
\widehat{\mathcal{Z}^{(\omega)}}(s)=&\ \sum_{n\geq0}\frac{1}{s+n \xi} \prod_{0\leq k\leq n-1} \left(\frac{\varrho}{\psi(s+k\xi)}\right).
\end{align}

Inverting the above Laplace transforms allows us to identify the $\omega$-scale functions for particular
L\'evy processes.

{\bf Drift minus compound Poisson process: $X_t=\mu t- \sum_{i=1}^{N_t}U_i$.}
Let $N_t$ be a Poisson process with intensity $\vartheta$ and $\{U_i\}_{i=1}^\infty$ 
a sequence of i.i.d. exponentially random variables, independent of $N$, with parameter $\rho$.  
Then, $\psi(s)=\frac{s+\varsigma}{s+\rho} \mu s$, where $\varsigma=\rho- \frac{\vartheta}{\mu}$. If we assume that $\varsigma, \rho \notin\{n\xi, n\in \mathbb{Z}\}$, then 
we have:
\begin{align*}
\mathcal{W}^{(\omega)}(x)=&\ \frac{\rho}{\varsigma\mu}\cdot {_1 F_1}\left(\frac{\xi+\rho}{\xi}, \frac{\xi+ \varsigma}{\xi}; \frac{\varrho}{\mu\xi}\right)\cdot {_1 F_1}\left(\frac{\xi-\rho}{\xi}, \frac{\xi-\varsigma}{\xi}; \frac{-\varrho}{\mu\xi} e^{-\xi x}\right)\\
&+ \frac{\varsigma-\rho}{\varsigma\mu} e^{-\varsigma x}\cdot {_1 F_1}\left(\frac{\xi+\rho-\varsigma}{\xi}, \frac{\xi-\varsigma}{\xi}; \frac{\varrho}{\mu\xi}\right)
 \cdot {_1 F_1}\left(\frac{\xi+ \varsigma- \rho}{\xi}, \frac{\xi +\varsigma}{\xi}; \frac{-\varrho}{\mu\xi} e^{-\xi x}\right),\\
\mathcal{Z}^{(\omega)}(x)=&\ {_1 F_1}\left(\frac{\rho}{\xi},\frac{\varsigma}{\xi}; \frac{\varrho}{\mu\xi}\right)\cdot {_1 F_1}\left(\frac{\xi-\rho}{\xi}, \frac{\xi-\varsigma}{\xi}; \frac{-\varrho}{\mu\xi} e^{-\xi x}\right)\\
&+ \frac{(\rho-\varsigma)\varrho}{\varsigma\mu(\xi-\varsigma)} e^{-\varsigma x}\cdot {_1 F_1}\left(\frac{\xi+\theta-\varsigma}{\xi}, \frac{\xi-\varsigma}{\xi}; \frac{\varrho}{\mu\xi}\right)
 \cdot {_1 F_1}\left(\frac{\xi+\varsigma-\rho}{\xi}, \frac{\xi+\varsigma}{\xi}; \frac{-\varrho}{\mu\xi} e^{-\xi x}\right),
\end{align*}
where $\D {_1 F_1}(a,b; z)=\sum_{n\geq0} \frac{\Gamma(b) \Gamma(a+n)}{\Gamma(a) \Gamma(b+n)} \frac{z^n}{n!}$ is the Kummer's confluent hypergeometric function.\\

{\bf Linear Brownian motion: $X_t=\sigma B_t+\mu t$.}
In this case $\psi(s)=D s (s+R)$ where $\D D=\frac{1}{2}\sigma^2$ and $\D R=\frac{2\mu}{\sigma^{2}}$. If we assume that $R\notin\{n \xi, n\in \mathbb{Z}\}$, then
\begin{align}
\mathcal{W}^{(\omega)}(x)=&\
\frac{\Gamma(1+\alpha)\Gamma(1-\alpha)}{D R}
e^{-R x/2}\cdot \left[I_{\alpha}(\beta)I_{-\alpha}(\beta e^{-\xi x/2})- I_{-\alpha}(\beta) I_{\alpha}(\beta e^{-\xi x/2})\right],
\non\\
\mathcal{Z}^{(\omega)}(x)=&\ \frac{\sqrt{\varrho}\, \Gamma(1+\alpha)\Gamma(1-\alpha)}{R\sqrt{D}} e^{-Rx/2}\cdot \left[  I_{\alpha-1}\left(\beta\right)  I_{-\alpha}\left(
\beta e^{-\xi x/2}\right)-I_{1-\alpha}\left(\beta\right) I_{\alpha}\left(\beta e^{-\xi x/2}\right)\right],\non
\end{align}
where
\[
\alpha=\frac{R}{\xi}\notin \mathbb{N},\quad
\beta=\frac{2\sqrt{\varrho}}{\xi \sqrt{D}}
\]
and
$\D I_{\theta}(z)=\sum_{n\geq0} \frac{1} {n! \Gamma(\theta+n+1)}\left(\frac{z}{2}\right)^{2n+\theta}$ is the modified Bessel function of first type.

The above results can be used to reproduce another interesting observation. Note that, from the above representations of the $\omega$ scale functions we have that:
\begin{eqnarray*}
\mathcal{W}^{(\omega)}(x,y)&=&\frac{\Gamma(1-\alpha)\Gamma(1+\alpha)}{D R}
e^{-R x/2} e^{\xi \alpha y/2}\cdot
\left[I_{\alpha}(\beta e^{-\xi y/2})I_{-\alpha}(\beta e^{-\xi x/2})- I_{-\alpha}(\beta e^{-\xi y/2}) I_{\alpha}(\beta e^{-\xi x/2})\right].
\end{eqnarray*}
Now, recalling Corollary \ref{cor:2} and taking the limits $y \rightarrow -\infty$ and then $c \rightarrow \infty$ in the above equation, we obtain
\begin{equation}
\Em_{x}\left[\exp\left(-\varrho \int_{0}^{\infty} e^{-\xi (\sigma B_{t}+\mu t)}\,dt\right)\right]= \frac{2}{\Gamma(\alpha)} \left(\frac{\beta e^{-\xi x/2}}{2}\right)^{\alpha} K_{\alpha}(\beta e^{-\xi x/2}),
\end{equation}
where $K_{v}(z)$ is a modified Bessel function of the first kind.
In deriving the above identity, we also used the following asymptotics:
\begin{align}
&\ I_{v}(z)\sim \frac{e^{z}}{\sqrt{2\pi z}}\left[1+ \sum_{n\geq 1} (-1)^{n} \frac{\prod_{k=1}^{n} (4v^{2}-k^{2})}{n! (8z)^{n}}\right] && \text{as $|z|\to\infty$, $|\arg z|<\frac{\pi}{2}$},\\
&\ K_{v}(z)=\frac{\pi}{2} \frac{I_{-v}(z)-I_{v}(z)}{\sin(v\pi)} \sim \frac{1}{2} \Gamma(v) (\frac{1}{2}z)^{-v} &&  \text{as $z\to0$, $\Re v>0$};
\end{align}
see \cite[9.7.1 and 9.6.9]{Abramowitz1972Handbook}.
Taking $x=0$, $\xi=2$ and $\sigma=1$ recovers a well-known identity in law:
\[2  \int_{0}^{\infty} e^{-2(B_{s}+\mu s)}\,ds \stackrel{d}{=} (\gamma_{\mu})^{-1},\]
where $\gamma_{\mu}$ is a gamma variable with index $\mu$; see \cite{Bertoin2005:EF:survey}.

\subsection{Step $\omega$-function}
In this example we will consider
a positive step function $\omega(x)=p_0 + \sum_{j=1}^{n} (p_{j}-p_{j-1}) \mathbf{1}(x>x_{j})$ for
fixed $n$, fixed sequence $\{p_j\}_{j=0}^n$ and finite division of the real line $x_{k+1}\geq x_{k}$, for $k=1,2,\cdots, n-1$.

In the following, we will prove that
\[\wo(x,y)=\w_{n}(x,y)\quad \text{for}\quad x,y\in\mathbb{R},\]
where $\w_{n}(x,y)$ is defined recursively by 
\begin{align}
\w_{k+1}(x,y)=\w_{k}(x,y)+ (p_{k+1}-p_{k}) \int_{x_{k+1}}^{x} W^{(p_{k+1})}(x-z)\w_{k}(z,y)\,dz,\label{eqn:5}
\end{align}
where $\w_{0}(x,y)=W^{(p_{0})}(x-y)$.

Denote $\omega_{k}(x):=p_0+ \sum_{j=1}^{k} (p_{j}-p_{j-1}) \mathbf{1}(x>x_{j})$ with $\omega_{0}(z)=p_{0}$
and let $\w_{k}(\cdot,\cdot)$ be the $\omega$-scale function with respect to $\omega_{k}(\cdot)$.
In order to prove (\ref{eqn:5}) holds, note that from (\ref{eqn:Ww:xy:2}), we have 
\begin{align}
\w_{k}(x,y)- W^{(p_{k+1})}(x-y)
=&\ \int_{y}^{x} W^{(p_{k+1})}(x-z)(\omega_{k}(z)-p_{k+1})\w_{k}(z,y)\,dz,
\label{eqn:3}\\
\w_{k+1}(x,y)- W^{(p_{k+1})}(x-y)
=&\ \int_{y}^{x} W^{(p_{k+1})}(x-z)(\omega_{k+1}(z)-p_{k+1})\w_{k+1}(z,y)\,dz.
\label{eqn:4}
\end{align}

Then, by observing that $\omega_{k+1}(z)-p_{k+1}=0$, for $z\geq x_{k+1}$, and $\omega_{k+1}(z)=\omega_{k}(z)$, for $z \leq x_{k+1}$, 
from Lemma \ref{lem:equation} we can conclude that
$\wo_{k+1}(z,y)=\wo_{k}(z,y)$ for $z<x_{k+1}$. Thus, from \eqref{eqn:4} it follows that
\begin{align*}
\w_{k+1}(x,y)- W^{(p_{k+1})}(x-y)
=&\ \int_{y}^{x} W^{(p_{k+1})}(x-z)(\omega_{k+1}(z)-p_{k+1})\w_{k+1}(z,y)\,dz\\
=&\ \int_{y}^{x_{k+1}} W^{(p_{k+1})}(x-z)(\omega_{k+1}(z)-p_{k+1})\w_{k+1}(z,y)\,dz\\
=&\ \int_{y}^{x_{k+1}}W^{(p_{k+1})}(x-z)(\omega_{k}(z)-p_{k+1})\w_{k}(z,y)\,dz.
\end{align*}
Comparing this identity with \eqref{eqn:3} produces:
\[
\w_{k+1}(x,y)=\w_{k}(x,y)- \int_{x_{k+1}}^{x} W^{(p_{k+1})}(x-z)(\omega_{k}(z)-p_{k+1})\w_{k}(z,y)\,dz
\]
which implies \eqref{eqn:5} since $(\omega_{k}(z)-p_{k+1})=p_{k}-p_{k+1}$, for $z>x_{k+1}$.

Similar considerations can produce the following
representation of the second $\omega$-scale function:
\[\zo(x,y)=\z_{n}(x,y),\quad \text{for}\ x\geq y,\]
where $\z_{0}(x,y)=Z^{(p_{0})}(x-y)$ and
\begin{align*}
\z_{k+1}(x,y)=\z_{k}(x,y)+ (p_{k+1}-p_{k}) \int_{x_{k+1}}^{x} W^{(p_{k+1})}(x-z)\z_{k}(z,y)\,dz.\label{eqn:6}
\end{align*}

\section{Some preliminary facts}\label{sec:facts}

To evaluate $\mathcal{A}$ and $\mathcal{B}$ in Theorem \ref{thm:main}, we need
the following classical results which can be found e.g. in \cite[Chap. 8]{Kyprianou2014:book:levy}.

\begin{prop}\label{prop:levy} For $0<x<c$ and  $q\geq0$ we have:
\begin{align*}
\Em\left[e^{-q\tau_c^{+}}\right]=&\ e^{-\Phi(q) c}, &
\Em_x\left[e^{-q\tau_c^{+}}; \tau_c^{+}<\tau_0^{-}\right]=&\ \frac{W^{(q)}(x)}{W^{(q)}(c)},\\
\Em_x\left[e^{-q\tau_0^{-}}\right]=&\ Z^{(q)}(x)-\frac{q}{\Phi(q)}W^{(q)}(x), &
\Em_x\left[e^{-q\tau_0^{-}}; \tau_0^{-}<\tau_c^{+}\right]=&\ Z^{(q)}(x)- Z^{(q)}(c)\frac{W^{(q)}(x)}{W^{(q)}(c)}.
\end{align*}
Moreover,
\begin{align*}
R^{(q)}(x,dy):=&\ \int_0^{\infty} e^{-q t} \Pm_x(X_t\in dy; t<\tau_0^{-})\,dt= \left(e^{-\Phi(q)y} W^{(q)}(x)- W^{(q)}(x-y)\right)\,dy,\\
U^{(q)}(x,dy):=&\ \int_0^{\infty} e^{-q t} \Pm_x(X_t\in\,dy; t<\tau_0^{-}\wedge \tau_c^{+})\,dt= \left(\frac{W^{(q)}(x)W^{(q)}(c-y)}{W^{(q)}(c)}-W^{(q)}(x-y)\right)\,dy.
\end{align*}
\end{prop}

Another important proposition was proved in \cite{Pistorius2004:passagetime:reflect} and it will be heavily used in this paper
when we work with the reflected processes $Y_t$ and $\widehat{Y}_t$.
\begin{prop}\label{prop:refl} Let $x\in[0,a]$ and $q\geq 0$. Then
\begin{align*}
\Em_{x}\left[e^{-qT_a}\right]=&\ Z^{(q)}(x)/Z^{(q)}(a),\\
\Em_{-x}\left[e^{-q\widehat{T}_{a}}\right]=&\ Z^{(q)}(a-x)- q W^{(q)}(a-x) W^ {(q)}(a)/{W^{(q)\prime}_{+}}(a).
\end{align*}
Moreover,
\begin{align*}
L^{(q)}f(x):=&\ \int_0^{\infty} e^{-q t} \Em_{x}\left[f(Y_t), t<T_a\right]\,dt\non\\
=&\ \int_{[0,a)} f(y)\left(\frac{Z^{(q)}(x)}{Z^{(q)}(a)} W^{(q)}(a-y)- W^{(q)}(x-y)\right)\,dy,\\
\widehat{L}^{(q)}f(x):=&\ \int_0^{\infty} e^{-qt} \Em_{-x}\left[f(\widehat{Y}_t), t<\widehat{T}_a\right]\,dt\non\\
=&\ \int_{[0,a)} f(y) \left( \frac{W^{(q)}(a-x)}{W^{(q)\prime}(a)} W^{(q)}(dy)- W^{(q)}(y-x)\,dy\right),
\end{align*}
where $W^{(q)}(dy)$ denotes the Stieltjes measure associated with $W^{(q)}(x)$, with possible
mass, $W^{(q)}(0)$, at zero.
\end{prop}

For the proofs of the main results given in Section \ref{mainresults},
we will use the following general lemma that could be of independent interest.
%
%

\begin{lem}\label{lem:key}
Let $Z=\{Z_t, t\geq0\}$ be a Markov process with the lifetime $\zeta$,
and $q$-resolvent measures and transition probabilities given, respectively, by 
\[
Q_tf(x)=\Em_{x}[f(Z_t); t<\zeta]\quad\text{and}\quad K^{(q)}f(x)=\int_0^{\infty} e^{-q t} Q_t f(x)\,dt,
\]
where $f$ is an nonnegative bounded continuous function on $\mathbb{R}$, such that $K^{(0)}f(x)<\infty$.
Let $\omega$ be a nonnegative and locally bounded function on $\mathbb{R}$ and $K^{(\omega)}$ be the $\omega$-type resolvent given by:
\begin{equation}
K^{(\omega)}f(x):=\int_0^{\infty} Q^{(\omega)}_tf(x)\,dt
\end{equation}
for
\[Q^{(\omega)}_tf(x):=\Em_{x}\left[\exp\left(-\int_0^{t} \omega(Z_s)\,ds\right)f(Z_t); t<\zeta\right].\]
Then, $K^{(\omega)}f(x)$ is finite and satisfies the following equality:
\begin{equation}\label{Komegaref}
K^{(\omega)}f(x)=\ K^{(0)}\left(f-\omega\cdot K^{(\omega)}f\right)(x).
\end{equation}
\end{lem}

\begin{proof}
The finiteness of $K^{(\omega)}f(x)$ comes from the fact that $K^{(\omega)}f(x)\leq K^{(0)}f(x)$.
Assume at the beginning that $\omega$ is bounded by some $\lambda>0$.
The arguments used here are similar to the arguments for the proof of Theorem \ref{thm:main} proved later.
Let $E=\{E_t, t\geq 0\}$ be a Poisson point process with a characteristic measure $\frac{1}{\lambda}\mathbf{1}_{(0,\lambda]}(y) dy \lambda dt$.
That is, $E=\{(T_i, M_i), i=1,2,\ldots\}$ is a (marked) Poisson jump process with jump intensity $\lambda$,
jump epochs $T_i$ and marks $M_i$ being uniformly distributed on $[0,\lambda]$.
Moreover, we define $E$ on a common probability space with $X$ and we construct $E$ to be independent of $X$.
Then
\begin{align*}
Q^{(\omega)}_tf(x)
=&\ \Em_x\left[f(Z_t); t<\zeta \text{\ and\ } M_i>\omega(Z_{T_i}) \text{\ for all\ }T_i<t\right]\\
=&\  \Em_x\left[f(Z_t); t<\zeta, T_1>t \right] + \int_0^t\lambda e^{-\lambda s} \Em_x\left[
Q^{(\omega)}_{t-s}f(Z_{s}), M_1>\omega(Z_s)\right]\,ds\non\\
=&\  \Em_x\left[e^{-\lambda t}f(Z_t); t<\zeta \right] + \int_0^t\lambda e^{-\lambda s} \Em_x\left[
\frac{(\lambda-\omega(Z_s))}{\lambda}Q^{(\omega)}_{t-s}f(Z_{s})\right]\,ds\non\\
=&\ Q_t^{(\lambda)}f(x)+ \int_0^{t} Q^{(\lambda)}_{s}\left((\lambda-\omega)Q^{(\omega)}_{t-s}f\right)(x)\,ds,
\end{align*}
where the superscript $\lambda$ denotes a counterpart for fixed $\omega(x)=\lambda$.
Thus
\begin{equation}
K^{(\omega)}f(x)=\int_0^{\infty} Q^{(\omega)}_tf(x)\,dt= K^{(\lambda)}f(x)+ K^{(\lambda)}\left((\lambda-\omega) K^{(\omega)}f\right)(x).\label{eqn:lem:inproof:2}
\end{equation}
Using the resolvent identity $\lambda K^{(0)}\circ K^{(\lambda)}=K^{(0)}-K^{(\lambda)}$, we have
\begin{align}
&\hspace{-0.5cm}\lambda K^{(0)}\left(K^{(\omega)}f\right)(x)\non\\
=&\ \lambda K^{(0)}\circ K^{(\lambda)}f(x)+ \lambda K^{(0)}\circ K^{(\lambda)}((\lambda-\omega)K^{(\omega)}f)(x)\non\\
=&\ (K^{(0)}-K^{(\lambda)})f(x)+ (K^{(0)}-K^{(\lambda)})((\lambda-\omega)K^{(\omega)}f)(x). \label{eqn:lem:inproof:3}
\end{align}
Comparing \eqref{eqn:lem:inproof:2} with \eqref{eqn:lem:inproof:3} completes the proof for the case of bounded $\omega$.
The general case follows from the classical limiting arguments applied in \eqref{Komegaref}.
\end{proof}

We will now prove Lemma \ref{lem:equation}, that is that equation \eqref{lem:eqn:1}:
\begin{equation*}
H(x)=h(x)+ \int_0^{x} W(x-y) \omega(y) H(y)\,dy
\end{equation*}
admits a unique locally bounded solution $H(x)$. The method of the proof
is motivated by the proofs of other renewal-type equations.

\subsection{Proof of Lemma \ref{lem:equation}}
In this proof we follow the standard idea behind the proof of
the renewal theorem.
For $x<0$ the statement is obvious so let us consider only that $x\geq 0$. Note that,
it suffices to prove the existence of the solution of equation
(\ref{lem:eqn:1}) on $[0,x_{0}]$, for any fixed $x_{0}>0$,
since an increase of $x_0$ must produce the same solution. In fact, the solution
does not depend on $x_0$.

Let $\varpi$ be an upper bound of $|\omega(y)|$ on $[0,x_0]$ and $s_0$ satisfies $\psi(s_0)>2\varpi$.
To prove the uniqueness of solution of \eqref{lem:eqn:1}, we will show that $H(x)=0$ is the only solution of
\[
H(x)= \int_0^{x} W(x-y)\omega(y)H(y)\,dy, \quad \forall x\in[0,x_{0}].
\]
First, note that
\[
e^{-s_0 x}|H(x)|\leq\ \int_0^{x} e^{-s_0 (x-y)}|H(x-y)| \left( \varpi e^{-s_{0}\cdot y} W(y)\right) \,dy.
\]
Then, since $\D \int_0^{\infty} \varpi e^{-s_0 y} W(y)\,dy=\frac{\varpi}{\psi(s_0)}<\frac{1}{2}$,
the classical renewal-type arguments (see e.g. \cite[Thm.~5.2.4, p.~146]{Apq}) give
$e^{-s_0 x}|H(x)|= 0$ on $[0,x_0]$ which proves the uniqueness.

Now, let us introduce $\{ \mathcal{G}^{(n)}, H_n\}_{n\geq 1}$ in the following way:
\begin{align}
&\ \mathcal{G}f(x)=\mathcal{G}^{(1)}f(x)=\int_0^{x} e^{-s_0 (x-y)}W(x-y)\omega(y)f(y)\,dy, \quad  \mathcal{G}^{(n)}f(x)=\mathcal{G}(\mathcal{G}^{(n-1)}f)(x),\nonumber\\
&\ H_0(x)=e^{-s_0 x} h(x),\quad  H_{n+1}(x)=H_0(x)+ \mathcal{G}H_n(x).\label{reccH}
\end{align}
Then, $\mathcal{G}$ is a linear operator such that $\D |\mathcal{G}f(x)|\leq \frac{1}{2} \sup_{y\in [0,x_0]} |f(y)|$ for $x\in[0,x_0]$.
Thus, for $m>n$, we  have
\[
|H_m(x)-H_n(x)|= |\sum_{k=n+1}^{m}\mathcal{G}^{(k)}H_0(x)|\leq 2^{-n} \sup_{y\in[0, x_0]} |H_0(y)|.
\]
Therefore, $\{( H_n(x), x\in [0, x_0]) \}_{n\geq 0}$ is a Cauchy sequence which admits a limit $\tilde{H}$ on $[0,x_0]$ satisfying:
\[
\tilde{H}(x)=H_0(x)+ \mathcal{G}\tilde{H}(x)= e^{-s_0 x} h(x)+ \int_0^{x} e^{-s_0 (x-y)} W(x-y)\omega(y) \tilde{H}(y)\,dy.
\]
It follows that $H(x)=e^{s_0 x} \tilde{H}(x)$ is the desired function satisfying \eqref{lem:eqn:1}.

To prove (\ref{lem:eqn:2}) we convolute both sides of \eqref{lem:eqn:1} with $W^{(\delta)}$ to get:
\begin{align*}
\delta W^{(\delta)}*H(x)&\ = \delta h*W^{(\delta)}(x)+ \delta (\omega H)*W*W^{(\delta)}(x)\\
&\ = \delta (h*W^{(\delta)})(x)+ (\omega H)*(W^{(\delta)}-W)(x),
\end{align*}
where we used the identity $W^{(\delta)}-W=\delta W^{(\delta)}*W$, given in (\ref{falka}).
Using \eqref{lem:eqn:1} we obtain
\[
H(x)=h(x)+ \delta h*W^{(\delta)}(x)+ \left((\omega-\delta)H\right)*W^{(\delta)}(x),
\]
which completes the proof of the second assertion of the lemma.
Finally, the monotonicity of the solution $H$ with respect to the function $\omega$
follows from the iterative construction (\ref{reccH}) of $H$.

\halmos

\section{Proofs of exit identities}\label{Proofs}

\subsection{Proof of Theorem \ref{thm:main}: formula \eqref{ans:A}}
For $x<0$ the statement is straightforward since
$\mathcal{W}^{(\omega)}(x)=W(x)=0$.
For $x\geq 0$, applying now the Markov property of $X$ at $\tau_y^{+}$ and using the fact that
$X$ has no positive jumps, we have
\begin{equation}
\mathcal{A}(x,z)= \mathcal{A}(x,y) \mathcal{A}(y,z),\label{relation:Markov}
\end{equation}
for all $z\geq y\geq x\geq0$.

We will use the same crucial idea of constructing of particular Poisson process introduced already in the proof of Lemma \ref{lem:key}.
At the beginning let $\omega$ be a bounded function and $\lambda$ be its arbitrary upper bound.
Define $E=\{E_t, t\geq 0\}$ to be a Poisson point process with a characteristic measure $\frac{1}{\lambda}\mathbf{1}_{(0,\lambda]}(y) dy \lambda dt$.
That is, $E=\{(T_i, M_i), i=1,2,\ldots\}$ is a (marked) Poisson jump process with jump intensity $\lambda$,
jump epochs $T_i$ and marks $M_i$ being uniformly distributed on $[0,\lambda]$.
Moreover, we define $E$ on a common probability space with $X$ and we construct $E$ to be independent of $X$.
The crucial observation is that:
\begin{equation}
\mathcal{A}(x,c)=\Pm_x\left[\left\{
M_i>\omega(X_{T_i})\quad \text{for all}\quad T_i< \tau_c^{+}\right\}\cap \left\{\tau_c^{+}<\tau_0^{-}\right\}\right].
\end{equation}
Then, by Proposition \ref{prop:levy} and
the Markov property of $X$
we have:
\begin{align*}
\mathcal{A}(x, c)=
&\  \Pm_x(T_1>\tau_c^{+}, \tau_c^{+}<\tau_0^{-})+ \Em_x\left[\mathcal{A}(X_{T_1}, c); T_1<\tau_0^{-}\wedge\tau_c^{+}, M_1>\omega(X_{T_1})\right]\non\\
=&\  \frac{W^{(\lambda)}(x)}{W^{(\lambda)}(c)}+ \int_0^{c} \left(\frac{W^{(\lambda)}(x) W^{(\lambda)}(c-y)}{W^{(\lambda)}(c)} - W^{(\lambda)}(x-y)\right) (\lambda-\omega(y))\mathcal{A}(y, c)\,dy\non\\
=&\  \frac{W^{(\lambda)}(x)}{W^{(\lambda)}(c)}\left(1+ \int_0^{c} W^{(\lambda)}(c-y)(\lambda-\omega(y))\mathcal{A}(y,c)\,dy\right)\\
&\qquad - \int_0^{x} W^{(\lambda)}(x-y)(\lambda-\omega(y))\mathcal{A}(y,c)\,dy.
\end{align*}
Re-arranging the above equality and applying relation \eqref{relation:Markov} give:
\begin{align}
&\ \mathcal{A}(x, c)\left(1+ \int_{0}^{x} W^{(\lambda)}(x-y)(\lambda-\omega(y)) \mathcal{A}(y, x) \,dy \right)\non\\
=&\ \frac{W^{(\lambda)}(x)}{W^{(\lambda)}(c)}\left(1+ \int_0^{c} W^{(\lambda)}(c-y)(\lambda-\omega(y))\mathcal{A}(y,c)\,dy\right).\non
\end{align}
Now, taking
\begin{equation}\label{defwnowa}
\mathcal{W}^{(\omega)}(x)= W^{(\lambda)}(x)\left(1+ \int_{0}^{x} W^{(\lambda)}(x-y)(\lambda-\omega(y)) \mathcal{A}(y, x) \,dy \right)^{-1}\end{equation}
gives the required identity $\mathcal{A}(x, c)=\mathcal{W}^{(\omega)}(x)/\mathcal{W}^{(\omega)}(c)$ for all $0\leq x\leq c$.

Replacing $\mathcal{A}(y,x)$ by $\mathcal{W}^{(\omega)}(y)/\mathcal{W}^{(\omega)}(x)$ in \eqref{defwnowa} will produce:
\begin{align}
W^{(\lambda)}(x)=&\ \mathcal{W}^{(\omega)}(x)\left(1+ \int_0^{x} W^{(\lambda)}(x-y)(\lambda-\omega(y)) \frac{\mathcal{W}^{(\omega)}(y)}{\mathcal{W}^{(\omega)}(x)}\,dy\right)\non\\
=&\ \mathcal{W}^{(\omega)}(x)+ \int_0^{x} W^{(\lambda)}(x-y) (\lambda-\omega(y))\mathcal{W}^{(\omega)}(y)\,dy,\non
\end{align}
which by applying Lemma \ref{lem:equation} and identity (\ref{falka}) shows that $\mathcal{W}^{(\omega)}(x)$,
given in (\ref{defwnowa}), satisfies equation (\ref{eqn:Ww}).

Now, let $\omega$ be a general locally bounded function. Note that from Lemma \ref{lem:equation}
one can conclude that
$\mathcal{W}^{(\omega)}$ is an increasing functional of $\omega$, that is if $\omega_1(x)\geq \omega_2(x)$ for all $x\geq0$ then $\mathcal{W}^{(\omega_1)}(x)\geq \mathcal{W}^{(\omega_2)}(x)$.
Let $\Psi>0$, $\omega_\Psi(x)=\omega(x)\wedge \Psi$ and denote
\[
\mathcal{A}_\Psi(x,b)=\Em_x\left[\exp\left(-\int_0^{\tau_c^{+}} \omega_\Psi(X_t)\,dt\right); \tau_c^{+}<\tau_0^{-}\right].
\]
Then, for $\mathcal{W}^{(\omega_{\Psi})}(x)$ satisfying
\[
\mathcal{W}^{(\omega_\Psi)}(x)= W(x)+ \int_0^{x} W(x-y) \omega_\Psi(y) \mathcal{W}^{(\omega_{\Psi})}(y)\,dy
\]
we have:
\[
\mathcal{A}_\Psi(x, b)=\frac{\mathcal{W}^{(\omega_\Psi)}(x)}{\mathcal{W}^{(\omega_\Psi)}(b)}.
\]
Finally, letting $\Psi\rightarrow \infty$ and using the monotone convergence theorem
completes the proof of Theorem \ref{thm:main}, eq. \eqref{ans:A}.

\halmos

\subsection{Proof of Theorem \ref{thm:main}: formula \eqref{ans:B}}
For $x<0$ the statement is straightforward since $\mathcal{W}^{(\omega)}(x)=0$ and $\mathcal{Z}^{(\omega)}(x)=1$.
For $x\geq 0$, applying the Markov property at $\tau_c^{+}$ gives:
\begin{equation}\label{B}
\mathcal{B}(x, c)=B(x)- B(c) \frac{\mathcal{W}^{(\omega)}(x)}{\mathcal{W}^{(\omega)}(c)},
\end{equation}
where
\begin{equation}
B(x):=\Em_x\left[\exp\left(-\int_0^{\tau_{0}^{-}} \omega(X_t)\,dt\right)\right].
\end{equation}
As far as $\mathcal{B}(x,c)$ is concerned, only $\omega$ on $[0,c]$ matters. However, without loss of generality we can assume
that $\omega$ is well defined also on $(c,\infty)$. Then
\begin{align*}
1- B(x)
=&\ \Em_{x}\left[\int_0^{\tau_{0}^{-}} \omega(X_t) \exp\left(-\int_0^{t} \omega(X_s)\,ds\right)\,dt\right]\\
=&\ \int_0^{\infty} \Em_{x}\left[ \omega(X_t)\exp\left(-\int_0^{t} \omega(X_s)\,ds\right); t<\tau_0^{-}\right]\,dt\\
=&\ \int_0^{\infty}\left[\omega(y)-\omega(y)(1-B(y))\right] \left(e^{-\Phi(0) y} W(x)-W(x-y)\right)\,dy,
\end{align*}
where Lemma \ref{lem:key} and Proposition \ref{prop:levy} were applied in the last line. Then, straightforward calculations show
that $B$ satisfies the following equation:
\begin{equation}
B(x)= 1- c_B W(x)+ \int_0^{x} W(x-y)\omega(y) B(y)\,dy,
\end{equation}
for $\D c_B=\int_0^{\infty} e^{-\Phi(0) y} \omega(y) B(y)\,dy$.
Recalling that $\mathcal{W}^{(\omega)}$ and $\mathcal{Z}^{(\omega)}$ are the unique locally bounded solutions to
equations (\ref{eqn:Ww}) and (\ref{eqn:Zw}), respectively,
one can conclude that $B(x)=\mathcal{Z}^{(\omega)}(x)- c_B \mathcal{W}^{(\omega)}(x)$ and thus,
from equation (\ref{B}) it follows that
\begin{equation}
\mathcal{B}(x, c)=\mathcal{Z}^{(\omega)}(x)- \mathcal{Z}^{(\omega)}(c)\frac{\mathcal{W}^{(\omega)}(x)}{\mathcal{W}^{(\omega)}(c)},
\end{equation}
which completes the proof.

\halmos

\subsection{Proof of Corollary \ref{cor:one:down}}
We will prove only the second, nontrivial part of the statement.

If $c_{\mathcal{W}^{-1}(\infty)}^{-1}=\lim_{x\to\infty}\mathcal{W}^{(\omega)}(x):=\mathcal{W}^{(\omega)}(\infty)<\infty$, then $\Pm_{x}\left(\tau_0^{-}=\infty\right)\geq c_{\mathcal{W}^{-1}(\infty)}
\mathcal{W}^{(\omega)}(x)>0$ for any $x>0$. Thus,
$\lim_{x\to\infty}W(x):=W(\infty)<\infty$.
Recalling that $\mathcal{W}^{(\omega)}(x)$ satisfies
\begin{equation}
\mathcal{W}^{(\omega)}(x)= W(x)+ \int_0^{x} W(x-y) \omega(y) \mathcal{W}^{(\omega)}(y)\,dy
\end{equation}
and that $W(\cdot)$ is a monotone function, we can conclude that $\D \int_0^{\infty} \omega(y) \mathcal{W}^{(\omega)}(y)\,dy<\infty$.
From the monotonicity of $\mathcal{W}^{(\omega)}(\cdot)$, observed in the formula \eqref{ans:A},
it follows that $\D \int_{0}^{\infty} \omega(y)\,dy<\infty$.

On the other hand, if $\D \int_0^{\infty} \omega(y)\,dy<\infty$ and $W(\infty)<\infty$ (since $X_t\to \infty$ a.s.), then
\begin{equation}
\mathcal{W}^{(\omega)}(x)\leq W(\infty)+ W(\infty)\int_{0}^{x} \omega(y) \mathcal{W}^{(\omega)}(y)\,dy
\end{equation}
and the Gronwall's inequality gives:
\begin{equation}
\mathcal{W}^{(\omega)}(x)\leq  W(\infty) \exp\left(W(\infty) \int_0^{x} \omega(y)\,dy\right).
\end{equation}
This inequality shows the finiteness of $\D \lim_{x\to\infty} \mathcal{W}^{(\omega)}(x)$ and completes the proof.

\halmos

\subsection{Proof of Theorem \ref{thm:refl}: formula \eqref{ans:C}}
Applying the Markov property of $Y$ at $T_y$ and using fact that $Y$ is absent of positive jumps, we have
\begin{equation}
\mathcal{C}(x,z)=\mathcal{C}(x,y)\cdot \mathcal{C}(y,z),\label{relation:Markov:refl}
\end{equation}
for all $z\geq y\geq x\geq 0$.
Now, for $x\leq c$ we have
\begin{align*}
1-\mathcal{C}(x,c)=&\ \Em_{x}\left[1-\exp\left(-\int_0^{T_c} \omega(Y_t)\,dt\right)\right]\\
=&\ \int_0^{\infty} \Em_{x}\left[\omega(Y_t) \exp\left(-\int_0^{t}\omega(Y_s)\,ds\right); t<T_c\right]\,dt,
\end{align*}
which after applying \ref{lem:key}, Proposition \ref{prop:refl} and using relation \eqref{relation:Markov:refl}, gives
\begin{align}
\mathcal{C}(x,c)=&1- \int_0^{c} \omega(y) \mathcal{C}(y,c) \left( W(c-y)- W(x-y)\right)\,dy\nonumber\\
=& 1- \int_0^{c} W(c-y) \omega(y)\mathcal{C}(y,c)\,dy + \mathcal{C}(x,c) \int_0^{x} W(x-y) \omega(y)\mathcal{C}(y,x)\,dy.\label{2}
\end{align}
Thus, taking
\begin{equation}\label{3}
\mathcal{Z}^{(\omega)}(x)=\left(1- \int_0^{x} W(x-y)\omega(y)\mathcal{C}(y,x)\,dy\right)^{-1},
\end{equation}
gives the desired identity $\mathcal{C}(x,y)=\mathcal{Z}^{(\omega)}(x)/\mathcal{Z}^{(\omega)}(y)$ for $0\leq x\leq y$.

Finally, replacing $\mathcal{C}(y,x)$ by $\mathcal{Z}^{(\omega)}(y)/\mathcal{Z}^{(\omega)}(x)$ in \eqref{3} will produce equation
\eqref{eqn:Zw} for $\mathcal{Z}^{(\omega)}(\cdot)$ which completes the proof.

\halmos

For the proofs concerning the dual process $\widehat{Y}$, the derivatives of the
scale functions are needed, as well as those of the new scale functions.
Recalling the definition of $\mathcal{W}^{(\omega)}(x,z)$ in \eqref{eqn:Ww:xy}, we note that for $x> y$:
\begin{eqnarray}
\frac{\partial}{\partial x} \mathcal{W}^{(\omega)}(x, y)&=&\ W'(x- y)+ W(0)\omega(x) \mathcal{W}^{(\omega)}(x, y)
+ \int_{y}^{x} W'(x-z)\omega(z) \mathcal{W}^{(\omega)}(z,y)\,dz\non\\
&=&\ W'(x-y)+ \int_{0-}^{x-y} \omega(x-z) \mathcal{W}^{(\omega)}(x-z, y) W(dz).\label{eqn:Wf:dif}
\end{eqnarray}
Similarly, for $x>0$,
\begin{align}
\mathcal{W}^{(\omega)\prime}(x)=&\ W'(x)+ \int_{0-}^{x} \omega(x-y)\mathcal{W}^{(\omega)}(x-y)W(dy),\label{eqn:Ww:dif}\\
\mathcal{Z}^{(\omega)\prime}(x)=&\ \int_{0-}^{x} \omega(x-y)\mathcal{Z}^{(\omega)}(x-y) W(dy).\label{eqn:Zw:dif}
\end{align}

\subsection{Proof of Theorem \ref{thm:refl}: formula \eqref{ans:C:dual}}
From Lemma \ref{lem:key} and the Fubini theorem we have:
\begin{eqnarray*}
1- \widehat{\mathcal{C}}(x,c)
&=& \widehat{\Em}_{x}\left[ \int_0^{T_{c}} \omega(c- Y_t)\exp\left(-\int_0^{t} \omega(c- Y_s)\,ds\right)\,dt\right]\\
&=& \int_0^{\infty} \widehat{\Em}_{x}\left[ \omega(c-{Y}_t)\exp\left(-\int_0^{t} \omega(c- {Y}_s)\,ds\right); t< {T}_c\right]\,dt\\
&=& \widehat{L} \left( \omega(c-\cdot)- \omega(c-\cdot)(1-\widehat{\mathcal{C}}(\cdot, c))\right)(x)\\
&=& \int_{[0,c]} \omega(c-y) \widehat{\mathcal{C}}(y,c)\left(\frac{W(c-x)}{W'(c)}W(dy)- W(y-x)\,dy\right).
\end{eqnarray*}
Then, a change of of variable formula produces:
\begin{align*}
\widehat{\mathcal{C}}(c-x,c)=&\ 1- \int_{[0,c]} \omega(c-y) \widehat{\mathcal{C}}(y,c) \left(\frac{W(x)}{W'(c)}W(dy)- W(y+x-c)\,dy\right)\\
=&\ 1- c_{\widehat{C}}\cdot W(x)+ \int_0^{x} W(x-y) \omega(y)\widehat{\mathcal{C}}(c-y,c)\,dy, \label{eqn:C:dual:inproof:1}
\end{align*}
where
\begin{equation}
c_{\widehat{C}}=\frac{1}{W'(c)} \int_{[0,c]} \omega(c-y) \widehat{\mathcal{C}}(y,c) W(dy). \label{eqn:C:dual:inproof:2}
\end{equation}
Finally, using the definitions of $\mathcal{W}^{(\omega)}(\cdot)$, $\mathcal{Z}^{(\omega)}(\cdot)$ given in \eqref{eqn:Ww} and \eqref{eqn:Zw}, respectively, we have:
\begin{equation}
\widehat{\mathcal{C}}(c-x,c)=\mathcal{Z}^{(\omega)}(x)- c_{\widehat{C}}\mathcal{W}^{(\omega)}(x), \label{eqn:C:dual:inproof:1}
\end{equation}
which together with equation \eqref{eqn:C:dual:inproof:2} gives:
\begin{equation}
c_{\widehat{C}}= \frac{\int_{0-}^{c}\omega(c-y)\mathcal{Z}^{(\omega)}(c-y) W(dy)}{W'(c)+ \int_{0-}^{c} \omega(c-y)\mathcal{W}^{(\omega)}(c-y) W(dy)}=\frac{\mathcal{Z}^{(\omega)\prime}(c)}{\mathcal{W}^{(\omega)\prime}(c)}.
\end{equation}
This completes the proof.

\halmos

\subsection{Proof of Theorem \ref{cor:one:up}}
We will prove that
\begin{equation}\label{dodgranica}
\lim_{\gamma\to\infty}\frac{\mathcal{W}^{(\omega)}(x,-\gamma)}{\mathcal{W}^{(\omega)}(c,-\gamma)}=\frac{\mathcal{H}^{(\omega)}(x)}{\mathcal{H}^{(\omega)}(c)}.
\end{equation}
Then, \eqref{eqn:cor:1} will follow from Corollary \ref{cor:2}. Moreover,
\eqref{Xi} will follow from
Theorem \ref{thm:resolvent} since for $x,y\in(-\gamma,c]$,
\begin{eqnarray*}
\lefteqn{\int_0^{\infty} \Em_{x}\left[\exp\left(-\int_0^{t} \omega(X_s)\,ds\right); X_t\in\,dy, t<\tau_{-\gamma}^{-}\wedge\tau_{c}^{+}\right]\,dt}\\&&=
\left(\frac{\mathcal{W}^{(\omega)}(x,-\gamma)}{\mathcal{W}^{(\omega)}(c,-\gamma)} \mathcal{W}^{(\omega)}(c,y)- \mathcal{W}^{(\omega)}(x, y)\right)\,dy.\label{eqn:cor:inproof:1}
\end{eqnarray*}

To prove \eqref{dodgranica} first note that,
by Remark \ref{rmk:1} for $x\geq -\gamma$, we have:
\[
\mathcal{W}^{(\omega)}(x,-\gamma)= W^{(\phi)}(x+\gamma)+ \int_{0}^{x} W^{(\phi)}(x-z)(\omega(z)-\phi) \mathcal{W}^{(\omega)}(z, -\gamma)\,dz\label{eqn:cor:inproof:2}
\]
and observe also that for $u\in(-\gamma,0]$,
\[e^{-\Phi(\phi) \gamma} \mathcal{W}^{(\omega)}(u, -\gamma)
= e^{-\Phi(\phi) \gamma} W^{(\phi)}(u+\gamma),\]
where, by \cite[Chap. 9]{AndreasGerbershiubook}, we have 
\[\lim_{\gamma\to+\infty}e^{-\Phi(\phi) \gamma} W^{(\phi)}(u+\gamma)=\Phi'(\phi) e^{\Phi(\phi)u}.
\]
Moreover, for $x>0$, from Corollary \ref{cor:2} we have
\begin{equation}
\frac{\mathcal{W}^{(\omega)}(0,-\gamma)}{\mathcal{W}^{(\omega)}(x,-\gamma)}=\Em \left[\exp\left(-\int_0^{\tau_{x}^{+}} \omega(X_s)\,ds\right); \tau_{x}^{+}<\tau_{-\gamma}^{-}\right].
\end{equation}
This quantity is increasing with respect to $\gamma$ and hence the limit
\[
\lim_{\gamma\to+\infty} e^{-\Phi(\phi) \gamma} \mathcal{W}^{(\omega)}(x, -\gamma)= \Phi'(\phi) \left(\Em \left[\exp\left(-\int_0^{\tau_{x}^{+}} \omega(X_s)\,ds\right); \tau_{x}^{+}<\infty\right]\right)^{-1}
\]
is well-defined and finite for every $x\geq-\gamma$.
Finally, taking
\begin{equation}
\mathcal{H}^{(\omega)}(x)= \Phi'(\phi)^{-1}\lim_{\gamma\to\infty} e^{-\Phi(\phi) \gamma} \mathcal{W}^{(\omega)}(x,-\gamma)
\end{equation}
completes the proof of \eqref{dodgranica}. In order to show the above form for 
$\mathcal{H}^{(\omega)}(x)$ satisfies equation
\eqref{eqn:cor:H} note that
\begin{eqnarray*}\lefteqn{
e^{-\Phi(\phi) \gamma}\mathcal{W}^{(\omega)}(x, -\gamma)}\\&&=
e^{-\Phi(\phi) \gamma}\left(W^{(\phi)}(x+\gamma)+ \int_0^{x} W^{(\phi)}(x-y)(\omega(y)-\phi)\mathcal{W}^{(\omega)}(y, -\gamma)\,dy\right).
\end{eqnarray*}
Then, by taking the limit $\gamma \to\infty$ and applying the dominated convergence the result follows.

\halmos


\section{Proofs of resolvent identities}\label{Proofs: res}

\subsection{Proof of Theorem \ref{thm:resolvent}}
Let $f$ be a nonnegative bounded continuous function on $\mathbb{R}^{+}\cup\{0\}$.
Applying Lemma \ref{lem:key} and Proposition \ref{prop:levy} gives:
\begin{align*}
U^{(\omega)}f(x):=&\ \int_0^{\infty} \Em_{x}\left[f(X_t) \exp\left(-\int_0^{t} \omega(X_s)\,ds\right); t \leq \tau_0^{-}\wedge\tau_{c}^{+}\right]\,dt\\
=&\ \int_0^{c} \left(f(y)- \omega(y) U^{(\omega)}f(y)\right)U^{(0)}(x,dy)\\
=&\ c_{U}\cdot W(x)- W*f(x)+ \int_0^{x} W(x-y) \omega(y) U^{(\omega)}f(y)\,dy,
\end{align*}
where
\begin{equation}
c_{U}=\int_0^{c} \frac{W(c-y)}{W(c)} (f(y)- \omega(y) U^{(\omega)}f(y))\,dy. \label{eqn:U:inproof:1}
\end{equation}
We define the operator
\begin{equation}\label{mathcalR}
\mathcal{R}^{(\omega)}f(x):=\int_0^{x} f(y) \mathcal{W}^{(\omega)}(x, y)\,dy,\qquad x> 0,
\end{equation}
with $\mathcal{R}^{(\omega)}f(x)=0$ for $x\leq 0$. Then, the following identity holds true:
\begin{equation}\label{rmk:2}
\mathcal{R}^{(\omega)}f(x)= W*f(x)+ \int_0^{x} W(x-y)\omega(y) \mathcal{R}^{(\omega)}f(y)\,dy.
\end{equation}
Finally, from Lemma \ref{lem:equation} and equation \eqref{rmk:2} one can then conclude that
\[
U^{(\omega)}f(x)=c_{U}\cdot \mathcal{W}^{(\omega)}(x)- \mathcal{R}^{(\omega)}f(x)
\]
which together with the boundary condition $U^{(\omega)}f(c)=0$ gives formula \eqref{thm:resolvent:U}.

\halmos

\subsection{Proof of Theorem \ref{thm:resolv:refl}: resolvent $L^{(\omega)}$}
Let $f$ be a nonnegative bounded continuous function on $\mathbb{R}^{+}$. Applying Lemma \ref{lem:key} and Proposition \ref{prop:refl} gives
\begin{eqnarray*}
L^{(\omega)}f(x)&=& L\left(f-\omega L^{(\omega)}f\right)(x)\\
&=& \int_0^{c} (f(y)-\omega(y) L^{(\omega)}f(y)) W(c-y)\,dy- \int_0^{x} (f(y)-\omega(y) L^{(\omega)}f(y)) W(x-y)\,dy\\
&=& \int_0^{x} W(x-y)\omega(y) L^{(\omega)}f(y)\,dy- W*f(x)+ c_L,
\end{eqnarray*}
for some constant $c_L$. Thus, from Lemma \ref{lem:equation} and the definition of
the $\omega$-scale functions given in
\eqref{eqn:Zw} and \eqref{eqn:Ww:xy}, we have:
\[
L^{(\omega)}f(x)=c_L \cdot \mathcal{Z}^{(\omega)}(x)- \mathcal{R}^{(\omega)}f(x),
\]
where the operator $\mathcal{R}^{(\omega)}$ is defined in (\ref{mathcalR}).
Finally, employing the boundary condition $L^{(\omega)}f(c)=0$, we get
\begin{eqnarray*}
L^{(\omega)}f(x)&=& \frac{\mathcal{Z}^{(\omega)}(x)}{\mathcal{Z}^{(\omega)}(c)} \mathcal{R}^{(\omega)}f(c)- \mathcal{R}^{(\omega)}f(x)\\&=& \int_{[0,c]} f(y) \left(\frac{\mathcal{Z}^{(\omega)}(x)}{\mathcal{Z}^{(\omega)}(c)} \mathcal{W}^{(\omega)}(c,y)- \mathcal{W}^{(\omega)}(x,y)\right)\,dy,
\end{eqnarray*}
which completes the proof.

\halmos

\subsection{Proof of Theorem \ref{thm:resolv:refl}: resolvent $\widehat{L}^{(\omega)}$} Applying Lemma \ref{lem:key} and Proposition \ref{prop:refl} gives:
\begin{align*}
\widehat{L}^{(\omega)}f(x)=&\ \int_0^{\infty} \Em_{-x}\left[\exp\left(-\int_0^{t} \omega(c-\widehat{Y}_s)\,ds\right) f(\widehat{Y}_t); t<\widehat{T}_c\right]\,dt\\
=&\ \widehat{L}\left(f(\cdot)- \omega(c-\cdot) \widehat{L}^{(\omega)}f(\cdot)\right)(x)\\
=&\ \int_{0-}^{c} \left(f(y)-\omega(c-y) \widehat{L}^{(\omega)}f(y)\right) \left(\frac{W(c-x)}{W'(c)}W(dy)- W(y-x)\,dy\right)\\
=&\ c_{\widehat{L}} W(c-x)+ \int_0^{c} W(y-x) \omega(c-y)\widehat{L}^{(\omega)}f(y)\,dy- \int_0^{c} f(y) W(y-x)\,dy,
\end{align*}
where
\begin{equation}
c_{\widehat{L}}=W'(c)^{-1}\int_{0-}^{c} \left(f(y)-\omega(c-y) \widehat{L}^{(\omega)}f(y)\right) W(dy). \label{eqn:hL:inproof:2}
\end{equation}
Now, a change of variables formula produces:
\[
\widehat{L}^{(\omega)}f(c-x)= c_{\widehat{L}} \cdot W(x) - (W*f(c-\cdot))(x)+ \int_0^{x} W(x-y)\omega(y)\widehat{L}^{(\omega)}f(c-y)\,dy,
\]
which along with \eqref{eqn:Ww} and \eqref{eqn:Ww:xy}, for $\mathcal{R}^{(\omega)}$ defined in (\ref{mathcalR}), we get:
\begin{equation}
\widehat{L}^{(\omega)}f(c-x)=\ c_{\widehat{L}} \cdot \mathcal{W}^{(\omega)}(x)- (\mathcal{R}^{(\omega)}f(c-\cdot))(x), \qquad x\in[0,c].\label{eqn:hL:inproof:1}
\end{equation}
Substituting \eqref{eqn:hL:inproof:1} into \eqref{eqn:hL:inproof:2} gives:
\begin{eqnarray*}
\lefteqn{\left(W'(c)+ \int_{0-}^{c} \omega(c-y) \mathcal{W}^{(\omega)}(c-y) W(dy)\right) c_{\widehat{L}}}\\
&&=\ \int_{0-}^{c} \left(f(y)+ \omega(c-y) (\mathcal{R}^{(\omega)}f(c-\cdot))(c-y)\right)\,W(dy)\\
&&=\ f(0)W(0)+ \int_{0}^{c} f(c-z) \left(W'(c-z)+ \int_{0-}^{c}\omega(c-y) \mathcal{W}^{(\omega)}(c-y,z) \,W(dy)\right)\,dz.
\end{eqnarray*}
Now using equations \eqref{eqn:Wf:dif} and \eqref{eqn:Ww:dif} together with the fact that $W(0)=\mathcal{W}^{(\omega)}(0)$ we obtain:
\[
\mathcal{W}^{(\omega)\prime}(c) c_{\widehat{L}}=\mathcal{W}^{(\omega)}(0) f(0)+  \int_0^{c} \frac{\partial}{\partial x}\mathcal{W}^{(\omega)}(c,z) f(c-z)\,dz.
\]
Finally combining the above results, we have 
\begin{align*}
\widehat{L}^{(\omega)}f(x)=&\ f(0) \frac{ \mathcal{W}^{(\omega)}(0)\mathcal{W}^{(\omega)}(c-x)}{\mathcal{W}^{(\omega)\prime}(c)}\\
&\ +\int_0^{c} f(z) \left( \frac{\mathcal{W}^{(\omega)}(c-x)}{\mathcal{W}^{(\omega)\prime}(c)} \frac{\partial}{\partial x}\mathcal{W}^{(\omega)}(c,c-z)- \mathcal{W}^{(\omega)}(c-x,c-z)\right)\,dz,
\end{align*}
which completes the proof.

\halmos

\section*{Acknowledgements}
We are very grateful to Xiaowen Zhou for
suggesting to consider the step function $\omega$ and 
to Lewis Ramsden for many helpful remarks and comments.
The authors would like to
thank the anonymous reviewer for her/his useful suggestions that improved the quality of the paper.

\appendix

\bibliographystyle{imsart-nameyear}

\end{document}